\documentclass[12pt]{amsart}
\usepackage{cases}
\usepackage{cases}
\usepackage{cases}
\usepackage{mathrsfs}
\usepackage{amsfonts}
\usepackage{amscd}
\usepackage{tikz}
\usepackage{latexsym,amssymb}
\usepackage{a4wide}
\usepackage{epic,eepic,latexsym, amssymb, amscd, amsfonts, xypic, floatflt}
\usepackage{amsthm}
\usepackage{amsmath}
\usepackage{amscd}
\usepackage{graphicx}
\usepackage[mathscr]{eucal}
\usepackage{verbatim}

\input xy
\xyoption{all}
\numberwithin{equation}{section}
\begin{document}
\title[Cyclotomic expansion and volume conjecture for superpolynomials]{%
Cyclotomic expansion and volume conjecture for superpolynomials of colored HOMFLY-PT homology and colored Kauffman homology}
\author[Qingtao Chen]{Qingtao Chen}
\address{Department of Mathematics \\
ETH Zurich \\
8092 Zurich \\
Switzerland }
\email{qingtao.chen@math.ethz.ch}

\begin{abstract}
We first study superpolynomial associated to triply-graded reduced colored HOMFLY-PT homology. We propose conjectures of congruent relations and cyclotomic expansion for it. We prove conjecture of $N=1$ for torus knot case, through which we obtain the corresponding invariant $\alpha(T(m,n))=-(m-1)(n-1)/2$. This is closely related to the Milnor conjecture. Many examples including homologically thick knots and higher representations are also tested. Based on these examples, we further propose a conjecture that invariant $\alpha$ determined in cyclotomic expansion at $N=1$ is a lower bound for smooth 4-ball genus. According to the structure of cyclotomic expansion, we propose a volume conjecture for $SU(n)$ specialized superpolynomial associated to reduced colored HOMFLY homology. We also prove the figure eight case for this new volume conjecture.

Then we study superpolynomial associated to triply-graded reduced colored Kauffman homology. We propose a conjecture of cyclotomic expansion for it. Homologically thick examples and higher representations are tested. Finally we apply the same idea to the Heegaard-Floer knot homology and also obtain an expansion formula for all the examples we tested.
\end{abstract}

\maketitle

\theoremstyle{plain} \newtheorem{thm}{Theorem}[section] \newtheorem{theorem}[%
thm]{Theorem} \newtheorem{lemma}[thm]{Lemma} \newtheorem{corollary}[thm]{%
Corollary} \newtheorem{proposition}[thm]{Proposition} \newtheorem{conjecture}%
[thm]{Conjecture} \theoremstyle{definition} \newtheorem{remark}[thm]{Remark}
\newtheorem{remarks}[thm]{Remarks} \newtheorem{definition}[thm]{Definition}
\newtheorem{example}[thm]{Example}






\section{Introduction}

For the past 30 years, we witnessed many exciting developments in the area
of knot theory which has also been connected to many active areas in
mathematics and physics. Quantum invariants of knots and 3-manifolds was
pioneered by E. Witten's seminal paper \cite{Wit1} and was rigorously
defined by Reshetikhin-Turaev in \cite{RT}. About 15 years ago, M. Khovanov
\cite{Kho1} introduce the idea of categorification by illustrating an
example of categorification of the classical Jones polynomial. The reduced
Poincare polynomial of Khovanov's homology $\mathcal{P}(\mathcal{K};q,t)$
recovers the classical Jones polynomial $J(\mathcal{K};q)$ in the following
meaning
\begin{equation}
\mathcal{P}(\mathcal{K};q,-1)=J(\mathcal{K};q).
\end{equation}

He also showed that $\mathcal{P}(5_{1};q,-1)\neq \mathcal{P}(10_{132};q,-1)$
for knots $5_{1}$ and $10_{132}$, while they share the same Jones
polynomial, i.e. $J(5_{1};q)=J(10_{132};q)$. Then Khovanov-Rozansky \cite%
{KR1} generalize the categorification of Jones polynomial to the
categorification of the $sl(n)$ invariants, whose corresponding Poincare
polynomial $\mathcal{P}^{sl(n)}(\mathcal{K};q,t)$ recovers classical
HOMFLY-PT polynomial $P(\mathcal{K};a,q)$ with specialization $a=q^{n}$,
i.e. $\mathcal{P}^{sl(n)}(\mathcal{K};q,-1)=P(\mathcal{K};q^{n},q)$.

The idea of superpolynomial $\mathcal{P}(\mathcal{K};a,q,t)$ was introduced
in \cite{DGR} by Dunfield, Gukov and Rasmussen, which is a kind of
categorification and could recover both the classical HOMFLY-PT polynomial
and Alexander polynomial respectively, i.e. $\mathcal{P}(\mathcal{K}%
;q^{n},q,-1)=P(\mathcal{K};q^{n},q)$ and $\mathcal{P}(\mathcal{K}%
;-1,q,-1)=\Delta _{\mathcal{K}}(q^{2})$, where $\Delta _{\mathcal{K}}(q)$ is
the Alexander polynomial in the normal sense. This was further studied by
Khovanov-Rozansky in \cite{KR2}. It is a bit tricky that two theories
doesn't match directly.%
\begin{equation}
\mathcal{P}(\mathcal{K};q^{n},q,t)\neq \mathcal{P}^{sl(n)}(\mathcal{K};q,t),
\end{equation}

However, Dunfield, Gukov and Rasmussen argued \cite{DGR} superpolynomial $%
\mathcal{P}(\mathcal{K};a,q,t)$ could recover $\mathcal{P}^{sl(n)}(\mathcal{K%
};q,t)$ after certain differential $d_{n}$ involved. They further argued
\cite{DGR} that the specialized superpolynomial $\mathcal{P}(\mathcal{K}%
;t^{-1},q,t)$ could also recover the Poincare polynomial $HFK$($\mathcal{K}%
;q^{2},t$) of Heegaard-Floer knot homology $\widehat{HFK}_{i}$($\mathcal{K};s
$) under certain differential $d_{0}$. The Poincare polynomial $HFK$($%
\mathcal{K};q,t$) was given by%
\begin{equation}
HFK(\mathcal{K};q,t)\triangleq \underset{s,i\in
\mathbb{Z}
}{\sum }t^{i}q^{s}\widehat{HFK}_{i}(\mathcal{K};s),
\end{equation}

with condition

\begin{equation}
HFK(\mathcal{K};q,-1)=\Delta _{\mathcal{K}}(q).
\end{equation}

The Heegaard-Floer theory was independently constructed by Ozsv\'{a}th-Szab%
\'{o}\cite{OS2} and Rasmussen\cite{Ras}, which is another very active and
profound area.

\bigskip

Many people are interested in categorification of various invariants ranging
from classical invariants such HOMFLY-PT and Kauffman polynomials to their
colored version (with representation involved). Of course the theory of
superpolynomial become a very active area which attracts many mathematician
and physicists. More mathematical rigorous formulation of categorification
can be found in \cite{Web, Wu}

Congruent relations and cyclotomic expansion for colored $SU(n)$ invariants
was studied in papers \cite{CLPZ, CLZ} by joint works of the author with K.
Liu, P. Peng and S. Zhu. We get to know that congruent relations for quantum
invariants could imply certain cyclotomic expansion for these quantum
invariants

Our motivation of this paper is to have a correct point of view to study
congruent relations among these superpolynomials first.

There is a well-known result that Heegaard-Floer homology of an alternative
knot can be determined by a very simple method with only Alexander
polynomials and signature involved This result was proved by Ozsv\'{a}th-Szab%
\'{o} \cite{OS1}.

\begin{theorem}[Ozsv\'{a}th-Szab\'{o}]
Let $\mathcal{K}\subset S^{3}$ be an alternating knot with Alexander-Conway
polynomial $\Delta _{\mathcal{K}}(q)=\underset{s\in
\mathbb{Z}
}{\sum }a_{s}q^{s}$ and signature $\sigma =\sigma (\mathcal{K})$. Then we have%
\begin{equation}
\widehat{HFK}_{i}(\mathcal{K},s)=\left\{
\begin{array}{c}
\mathbb{Z}
^{|a_{s}|} \\
0%
\end{array}%
\right.
\begin{array}{c}
\text{if }i=s+\frac{\sigma }{2} \\
\text{otherwise}%
\end{array}%
\end{equation}
\end{theorem}

It was shown by C. Manolescu and P.S. Ozsv\'{a}th \cite{MO} that
quasi-alternating knots hold the same results.

Because one side of the triply-graded superpolynomials is also connected to
Heeggard-Floer knot homology under certain differential $d_{0}$. Thus it is
natural to propose a conjecture that under some $t$-grading shifting, we
could also obtain nice congruent relation properties just like the
non-categorified colored $SU(n)$ invariants\cite{CLPZ}. We first studied the
congruent relation properties for torus knots $T(2,2p+1)$, whose closed
formulas was obtained by H. Fuji, S. Gukov and P. Sulkowski in \cite{FGS1}.
After we did an intensive computation, we propose the following conjecture

\begin{conjecture}
The superpolynomial of triply-graded reduced colored HOMFLY-PT homology has the following
congruent relations%
\begin{eqnarray}
(-t)^{-Np}\mathcal{P}_{N}(T(2,2p+1);a,q,t) &\equiv &(-t)^{-kp}\mathcal{P}%
_{k}(T(2,2p+1);a,q,t)  \notag \\
&&mod(aq^{-1}+t^{-1}a^{-1}q)(t^{2}aq^{N+k}+t^{-1}a^{-1}q^{-N-k}),
\end{eqnarray}

where $\mathcal{P}_{N}(\mathcal{K};a,q,t)$ denote the superpolynomial of
triply-graded reduced colored HOMFLY-PT homology of a knot $\mathcal{K}$
with $N$-th symmetric power of the fundamental representation.
\end{conjecture}

As \cite{CLPZ, CLZ} suggest that there is always a cyclotomic expansion
behind such congruent relations. The reduced colored HOMFLY-PT
superpolynomial of the figure knot $4_{1}$ was obtained in (2.12) of \cite%
{FGS2} (original in \cite{IMMM}). We rearrange the expression of it in the
following way

\begin{equation}
\mathcal{P}_{N}(4_{1};a,q,t)=1+\underset{k=1}{\overset{N}{\sum }}\underset{%
i=1}{\overset{k}{\prod }}\left( \frac{\{N+1-i\}}{\{i\}}%
A_{i-2}(a,q,t)B_{N+i-1}(a,q,t)\right) ,  \label{CatCyclotomicfor4_1}
\end{equation}

where $A_{i}(a,q,t)=aq^{i}+t^{-1}a^{-1}q^{-i}$, $%
B_{i}(a,q,t)=t^{2}aq^{i}+t^{-1}a^{-1}q^{-i}$ and $\{p\}=q^{p}-q^{-p}$.

By setting $a=q^{2}$ and $t=-1$, we have $A_{i-2}(q^{2},q,-1)=\{i\}$, $%
B_{N+i-1}(q^{2},q,-1)=\{N+i+1\}$. Thus we could recover the original
cyclotomic expansion for the figure eight knot $4_{1}$ in the sense of
Harbiro \cite{Hab},

\begin{equation}
J_{N}(4_{1};q)=1+\underset{k=1}{\overset{N}{\sum }}\underset{i=1}{\overset{k}%
{\prod }}\{N+1-i\}\{N+1+i\},
\end{equation}

where $J_{N}(\mathcal{K};q)$ denotes the $N+1$ dimensional colored Jones
polynomial.

Inspired by (\ref{CatCyclotomicfor4_1}), we formulate the following
cyclotomic expansion formula for superpolynomial of triply-graded reduced
colored HOMFLY-PT homology.

\begin{conjecture}
For any knot $\mathcal{K}$, there exists an integer valued invariant $\alpha (\mathcal{K})\in
\mathbb{Z}
$, s.t. the superpolynomial $\mathcal{P}_{N}(\mathcal{K}%
;a,q,t)$ of triply-graded reduced colored HOMFLY-PT homology of a knot $%
\mathcal{K}$ has the following cyclotomic expansion formula%
\begin{equation}
(-t)^{N\alpha (\mathcal{K})}\mathcal{P}_{N}(\mathcal{K};a,q,t)=1+\underset{%
k=1}{\overset{N}{\sum }}H_{k}(\mathcal{K};a,q,t)\left( A_{-1}(a,q,t)\underset%
{i=1}{\overset{k}{\prod }}\left( \frac{\{N+1-i\}}{\{i\}}B_{N+i-1}(a,q,t)%
\right) \right)
\end{equation}

with coefficient functions $H_{k}(\mathcal{K};a,q,t)\in
\mathbb{Z}
\lbrack a^{\pm 1},q^{\pm 1},t^{\pm 1}]$, where $%
A_{i}(a,q,t)=aq^{i}+t^{-1}a^{-1}q^{-i}$, $%
B_{i}(a,q,t)=t^{2}aq^{i}+t^{-1}a^{-1}q^{-i}$ and $\{p\}=q^{p}-q^{-p}$.
\end{conjecture}

\begin{remark}
The above Conjecture-Definition for invariant $\alpha (\mathcal{K})$ should be
understood in this way. If the above conjecture of a knot $\mathcal{K}$ is true for $N=1$,
then $\alpha (\mathcal{K})$ is defined. The next level of the conjecture is for $N\geq
2$ by using the same $\alpha (\mathcal{K})$. In this way, $\alpha (\mathcal{K})$ is defined even
if the conjecture is only true for $N=1$.
\end{remark}

\begin{remark}
This Conjecture could recover Conj. 1.2.
\end{remark}

\begin{remark}
$H_{k}(\mathcal{K};a,q,t)$ is independent of $N$, which only depends
on knot $\mathcal{K}$ and integer $k$.
\end{remark}

\begin{remark}
As many examples shows, one can not find such a conjecture for Poincare polynomial of Khovanov's original homology.
This shows that superpolynomial has a nice property than Khovanov's polynomial in the sense of cyclotomic expansion.
A possible reason to explain this phenomenon is that the differential $d_2$ kills the additional terms when one reduce superpolynomial to obtain the Khovanov's polynomial.
\end{remark}

We have the following theorem for quasi-alternating knots.

\begin{theorem}
The Conjecture 1.3 (Conj. 2.3) of case $N=1$ is ture for any quasi-alternating knot $%
\mathcal{K}$. Furthermore, we have
\begin{equation}
\alpha (\mathcal{K})=-\frac{\sigma (\mathcal{K})}{2}
\end{equation}

and the following expansion
\begin{equation}
(-t)^{-\frac{\sigma (\mathcal{K})}{2}}\mathcal{P}_{1}(\mathcal{K}%
;a,q,t)=1+(aq^{-1}+t^{-1}a^{-1}q)(t^{2}aq+t^{-1}a^{-1}q^{-1})H_{1}(\mathcal{K%
};a,q,t).
\end{equation}
\end{theorem}

We also tested many homologically thick knots such as $10_{124}$, $10_{128}$%
, $10_{132}$, $10_{136}$, $10_{139}$, $10_{145}$, $10_{152}$, $10_{153}$, $%
10_{154}$ and $10_{161}$ to illustrate this conjecture as well as many
examples with higher representation.

\bigskip

Torus knots and torus links are studied completely in \cite{DMMSS}. Based on
these highly nontrivial computations, we are able to prove the following
theorem for all torus knots.

\begin{theorem}
For any coprime pair $(m,n)=1$, where $m<n$, cyclotomic expansion conjecture
(Conj. 1.3) of $N=1$ is true for torus knot $T(m,n)$ and we have $\alpha
(T(m,n))=-(m-1)(n-1)/2$.
\end{theorem}

Now we are considering a problem relating to the sliceness of a knot.

\begin{definition}
The smooth 4-ball genus $g_{4}(\mathcal{K})$ of a knot $\mathcal{K}$ is the
minimum genus of a surface smoothly embedded in the 4-ball $B^{4}$ with
boundary the knot. In particular, a knot $\mathcal{K}\subset S^{3}$ is
called smoothly slice if $g_{4}(\mathcal{K})=0$.
\end{definition}

\begin{remark}
The invariant $\alpha (T(m,n))=-(m-1)(n-1)/2$ suggest a very close relation
between the above theorem and the following Milnor Conjecture, which was
first proved by P. B. Kronheimer and T. S. Mrowka in \cite{KM}
\end{remark}

\begin{conjecture}[Milnor]
The smooth 4-ball genus of torus knot $T(m,n)$ is $(m-1)(n-1)/2$.
\end{conjecture}

Rasmussen \cite{Ras2} introduced a knot concordant invariant $s(\mathcal{K})$%
,\ which is a lower bound for the smooth 4-ball genus for knots in the
following sense.

\begin{theorem}[Rasmussen]
For any knot $\mathcal{K}\subset S^{3}$, we have the following relation%
\begin{equation}
|s(\mathcal{K})|\leq 2g_{4}(\mathcal{K}).
\end{equation}
\end{theorem}

In addition, Rasmussen again proved Milnor Conjecture by a purely
combinatorial method in \cite{Ras2}.

\bigskip

Based on all the knots we tested in tables in Appendix and proved via
theorem, we are able to propose the following conjecture.

\begin{conjecture}
The invariant $\alpha (\mathcal{K})$ (determined by cyclotomic expansion
conjecture (Conj 1.3 or Conj. 2.3) for $N=1$) is a lower bound for smooth
4-ball genus $g_{4}(\mathcal{K})$, i.e.%
\begin{equation}
|\alpha (\mathcal{K})|\leq g_{4}(\mathcal{K}).
\end{equation}
\end{conjecture}

\begin{remark}
For many knots we tested, it is identical to the Ozsv\'{a}th-Szab\'{o}'s $%
\tau $ invariant and Rasmussen's $s$ invariant.
\end{remark}

\bigskip

Inspired by the cyclotomic expansion formula, finally we propose volume
conjecture for $SU(n)$ specialized superpolynomials of HOMFLY homology as
follows

\begin{conjecture}[Volume Conjecture for $SU(n)$ specialized superpolynomial]

For any hyperbolic knot $\mathcal{K}$, we have%
\begin{equation*}
2\pi \underset{N\rightarrow \infty }{\lim }\frac{\log \mathcal{P}_{N}(%
\mathcal{K};q^{n},q,q^{-(N+n-1)})|_{q=e^{\frac{\pi \sqrt{-1}}{N+b}}}}{N+1}%
=Vol(S^{3}\backslash \mathcal{K})+\sqrt{-1}CS(S^{3}\backslash \mathcal{K}) \text{ }(mod\sqrt{-1}\pi ^{2}%
\mathbb{Z}
),
\end{equation*}

where $b\geq 1$ and $\frac{n-1-b}{2}$ is not a positive integer.
\end{conjecture}

\begin{remark}
The condition of this conjecture is
much more relaxed than condition of former Volume conjectures, because here $b$ can be
any larger integers. For example, original volume conjecture is only valid for $%
n=2$ and $b=1$, but this volume conjecture is valid for all positive integer $b$
with $n=2$.
\end{remark}

\begin{remark}
It will be interesting to know the relationship of between this volume conjecture and
the one proposed by Fuji, Gukov and Sulkowski in \cite{FGS1}, where they used categorified A-polynomials
of knots.
\end{remark}

We prove this volume conjecture for the case of figure eight knot $4_{1}$.

\begin{theorem}
The above volume conjecture is valid for figure eight knot $4_{1}$.
\end{theorem}

Then we directly studied cyclotomic expansion for superpolynomial $\mathcal{F%
}_{N}(\mathcal{K};a,q,t)$ of triply-graded reduced colored Kauffman homology
formulated by S. Gukov and J. Walcher in \cite{GW}. We obtain the similar
expansion conjecture.

\begin{conjecture}
For any knot $\mathcal{K}$, there exists an integer valued invariant $\beta (\mathcal{K})\in
\mathbb{Z}
$, s.t. the superpolynomial $\mathcal{F}_{N}(\mathcal{K};a,q,t)$
of triply-graded reduced colored Kauffman homology of a knot $\mathcal{K}$
has the following cyclotomic expansion formula%
\begin{equation}
(-t)^{N\beta (\mathcal{K})}\mathcal{F}_{N}(\mathcal{K};a^{2},q^{2},t)=1+%
\underset{k=1}{\overset{N}{\sum }}F_{k}(\mathcal{K};a,q,t)\left( A_{-1}(a,q,t)%
\underset{i=1}{\overset{k}{\prod }}\left( \frac{\{2(N+1-i)\}}{\{2i\}}%
B_{N+i-2}(a^{2},q^{2},t)\right) \right)
\end{equation}

with coefficient functions $F_{k}(\mathcal{K};a,q,t)\in
\mathbb{Z}
\lbrack a^{\pm 1},q^{\pm 1},t^{\pm 1}]$, where $%
A_{i}(a,q,t)=aq^{i}+t^{-1}a^{-1}q^{-i}$, $%
B_{i}(a,q,t)=t^{2}aq^{i}+t^{-1}a^{-1}q^{-i}$ and $\{p\}=q^{p}-q^{-p}$.

In particular, one further have $\frac{F_{1}(\mathcal{K};a,q,t)}{%
taq+t^{-1}a^{-1}q^{-1}}\in
\mathbb{Z}
\lbrack a^{\pm 1},q^{\pm 1},t^{\pm 1}].$
\end{conjecture}

\begin{remark}
The above Conjecture-Definition for invariant $\beta (\mathcal{K})$ should be
understood in this way. If the above conjecture of a knot $\mathcal{K}$ is true for $N=1$,
then $\beta (\mathcal{K})$ is defined. The next level of the conjecutre is for $N\geq
2$ by using the same $\beta (\mathcal{K})$. In this way, $\beta (\mathcal{K})$ is defined even
if the conjecture is only true for $N=1$.
\end{remark}

\begin{remark}
$F_{k}(\mathcal{K};a,q,t)$ is independent of $N$, which only depends
on knot $\mathcal{K}$ and integer $k$.
\end{remark}

We tested many examples which also involved homologically thick knot such as
$8_{19}$ and $9_{42}$ and higher representation of $3_{1}$ and $4_{1}$.

\bigskip

Because Alexander polynomial and Heegaard-Floer knot polynomials can be
deduced from HOMFLY-PT polynomial and superpolynomial of triply-graded
reduced uncolored HOMFLY-PT homology by setting $a=1$ and $a=t^{-1}$
respectively (under certain differential $d_{0}$ for superpolynomial case).

For many non-trivial examples we tested, we find the following expansion
formula for Poincare polynomial of Heegaard-Floer knot homology.

For any knot $\mathcal{K}$, there exists an integer valued invariant $\gamma
(\mathcal{K})\in
\mathbb{Z}
$, s.t. Poincare polynomial $HFK(\mathcal{K};q^{2},t)$ of Heegaard-Floer
knot homology of a knot $\mathcal{K}$ has the following expansion formula%
\begin{equation}
(-t)^{\gamma (\mathcal{K})}HFK(\mathcal{K};q^{2},t)=1+KF(\mathcal{K}%
;q,t)(q+t^{-1}q^{-1})^{2}
\end{equation}

with coefficient functions $KF(\mathcal{K};q,t)\in
\mathbb{Z}
\lbrack q^{\pm 1},t^{\pm 1}]$.

We test the above expression of homologically thick knots $8_{19}$, $9_{42}$%
, $10_{124}$, $10_{128}$, $10_{132}$, $10_{136}$, $10_{139}$, $10_{145}$, $%
10_{152}$, $10_{153}$, $10_{154}$, $10_{161}$ and $41$ homologically thick
knots up to 11 crossings. We also prove two examples of Whitehead double for
this expansion formula.

\bigskip

In section $1$, we discuss the superpolynomial associated to triply-graded
reduced colored HOMFLY-PT homology and argue the reason why we formulate the
cyclotomic expansion conjecture in this way. We tested a lot of examples for
the conjecture by using formulas from various references In addition, we
prove the torus knot $T(m,n)$ for $N=1$ case and show that the corresponding
invariants $\alpha (T(m,n))=-(m-1)(n-1)/2$. Finally we further propose a
conjecture that $|\alpha (\mathcal{K})|$ could serve as a lower bound for
the smooth 4-ball genus. In section $2$, we propose a volume conjecture for $%
SU(n)$ specialized superpolynomial of colored HOMFLY homology with an
emphasis on the motivation. We prove the figure eight knot for this
conjecture. In section $3$, we study the cyclotomic expansion for
superpolynomial associated to triply-graded reduced colored Kauffman
homology. Again we provide many supporting examples from the literatures. In
section $4$, we study an expansion formula for Poincare polynomial of
Heegaard-Floer knot homology. Many homologically thick knot such as
Whitehead doubles are provided.  In section $5$, we provide many supporting
examples to the conjectures proposed in the previous sections.

\bigskip

\textbf{Acknowledgements.} I would like to thank Nicolai Reshetikhin for
introducing me to the area of categorification in the summer of 2006, I
thank Kefeng Liu and Shengmao Zhu for long term collaboration and many
helpful discussion on project of various LMOV conjectures and congruent
skein relations. I also thank Giovanni Felder, Matthew Hedden, Stanislav
Jabuka, Andy Manion, Ciprian Manolescu, Jiajun Wang, Ben Webster and Hao Wu
for helpful discussion and email correspondence on quantum invariants,
categorification and superpolynomials. Finally I would like to thank Sergei
Gukov, Yi Ni and Krzysztof Karol Putyra for the lightening discussion who
told the author about the current situation in categorification and thank
Jun Murakami for many suggestions after he read the first version of this
paper. The research of the author is supported by the National Centre of
Competence in Research SwissMAP of the Swiss National Science Foundation.

\section{Superpolynomials of colored HOMFLY-PT homology}

After we did an intensive computation of torus knot $T(2,2p+1)$, we propose
the following conjecture of congruent relations,

\begin{conjecture}
The superpolynomial of triply-graded reduced colored HOMFLY-PT homology has the following
congruent relations for torus knot $T(2,2p+1)$%
\begin{eqnarray}
(-t)^{-Np}\mathcal{P}_{N}(T(2,2p+1);a,q,t) &\equiv &(-t)^{-kp}\mathcal{P}%
_{k}(T(2,2p+1);a,q,t)  \notag \\
&&mod(aq^{-1}+t^{-1}a^{-1}q)(t^{2}aq^{N+k}+t^{-1}a^{-1}q^{-N-k}),
\end{eqnarray}

where $\mathcal{P}_{N}(\mathcal{K};a,q,t)$ denote the superpolynomial of
triply-graded reduced colored HOMFLY-PT homology of a knot $\mathcal{K}$
with $N$-th symmetric power of the fundamental representation.
\end{conjecture}

\begin{remark}
By setting $a=q^{n}$ and $t=-1$, the above congruent relations reduced to%
\begin{equation}
J_{N}^{SU(n)}(T(2,2p+1);a,q,t)\equiv J_{k}^{SU(n)}(T(2,2p+1);a,q,t)\text{ }%
mod(q^{n-1}-q^{1-n})(q^{N+k+n}-q^{-N-k-n}),
\end{equation}%
where $J_{N}^{SU(n)}(\mathcal{K};a,q,t)$ denote a colored $SU(n)$ invariants
of a knot $\mathcal{K}$.

Again by setting $n=2$, the above congruent relations reduced to%
\begin{equation}
J_{N}(T(2,2p+1);a,q,t)\equiv J_{k}(T(2,2p+1);a,q,t)\text{ }%
mod(q-q^{-1})(q^{N+k+2}-q^{-N-k-2}),
\end{equation}%
where $J_{N}(\mathcal{K};a,q,t)$ is just the $N+1$ dimensional colored Jones polynomial of a
knot $\mathcal{K}$.
\end{remark}

The above two congruent relations appears in a joint work of the author with
K. Liu, P. Peng and S. Zhu \cite{CLPZ}. But these two congruent relations
obtained by reduction from the categorified one are actually weaker than
those in \cite{CLPZ}. It is somewhat mysterious that either the
categorification procedure or a general $a$ let the congruent relations loss
the $mod(q^{N-k}-q^{k-N})$ part compared to \cite{CLPZ}.

Inspired by (\ref{CatCyclotomicfor4_1}), we formulate the following
cyclotomic expansion formula for superpolynomial of triply-graded HOMFLY-PT
homology.

\begin{conjecture}
There exists an integer valued invariant $\alpha (\mathcal{K})\in
\mathbb{Z}
$, s.t. superpolynomial $\mathcal{P}_{N}(\mathcal{K}%
;a,q,t)$ of triply-graded reduced colored HOMFLY-PT homology of a knot $%
\mathcal{K}$ has the following cyclotomic expansion formula%
\begin{equation}
(-t)^{N\alpha (\mathcal{K})}\mathcal{P}_{N}(\mathcal{K};a,q,t)=1+\underset{%
k=1}{\overset{N}{\sum }}H_{k}(\mathcal{K};a,q,t)\left( A_{-1}(a,q,t)\underset%
{i=1}{\overset{k}{\prod }}\left( \frac{\{N+1-i\}}{\{i\}}B_{N+i-1}(a,q,t)%
\right) \right)
\end{equation}

with coefficient functions $H_{k}(\mathcal{K};a,q,t)\in
\mathbb{Z}
\lbrack a^{\pm 1},q^{\pm 1},t^{\pm 1}]$, where $%
A_{i}(a,q,t)=aq^{i}+t^{-1}a^{-1}q^{-i}$, $%
B_{i}(a,q,t)=t^{2}aq^{i}+t^{-1}a^{-1}q^{-i}$ and $\{p\}=q^{p}-q^{-p}$.
\end{conjecture}

\begin{remark}
The above Conjecture-Definition for invariant $\alpha (\mathcal{K})$ should be
understood in this way. If the above conjecture of a knot $\mathcal{K}$ is true for $N=1$,
then $\alpha (\mathcal{K})$ is defined. The next level of the conjecture is for $N\geq
2$ by using the same $\alpha (\mathcal{K})$. In this way, $\alpha (\mathcal{K})$ is defined even
if the conjecture is only true for $N=1$.
\end{remark}

\begin{remark}
$H_{k}(\mathcal{K};a,q,t)$ is independent of $N$, which only depends
on knot $\mathcal{K}$ and integer $k$. Because of the case of torus knot $%
T(2,5)$ etc, we can not make the conjecture to take $\underset{i=1}{\overset{%
k}{\prod }}\left( \frac{\{N+1-i\}}{\{i\}}A_{i-2}(a,q,t)B_{N+i-1}(a,q,t)%
\right) $ as the expansion basis, which is a tricky part of this conjecture.
\end{remark}

\begin{remark}
There is another cyclotomic expansion formulation of quadruply-graded
homology for 2-bridge knots and torus knots obtained in \cite{NO}, and an extension to link case in \cite{GNSSS}.
\end{remark}

\bigskip

For instance, we have the following expansion for $N=1$ and $2$.%
\begin{equation}
(-t)^{\alpha (\mathcal{K})}\mathcal{P}_{1}(\mathcal{K};a,q,t)=1+H_{1}(%
\mathcal{K};a,q,t)(aq^{-1}+t^{-1}a^{-1}q^{1})(t^{2}aq+t^{-1}a^{-1}q^{-1})
\end{equation}

and

\begin{eqnarray}
(-t)^{2\alpha (\mathcal{K})}\mathcal{P}_{2}(\mathcal{K};a,q,t) &=&1+H_{1}(%
\mathcal{K}%
;a,q,t)(aq^{-1}+t^{-1}a^{-1}q^{1})(q+q^{-1})(t^{2}aq^{2}+t^{-1}a^{-1}q^{-2})
\notag \\
&&+H_{2}(\mathcal{K}%
;a,q,t)(aq^{-1}+t^{-1}a^{-1}q^{1})(t^{2}aq^{2}+t^{-1}a^{-1}q^{-2})(t^{2}aq^{3}+t^{-1}a^{-1}q^{-3})
\end{eqnarray}

\bigskip

We have the following theorem for quasi-alternating knots.

\begin{theorem}
The Conjecture 1.3 (Conj. 2.3) of case $N=1$ is true for any quasi-alternating knot $%
\mathcal{K}$. Furthermore, we have
\begin{equation}
\alpha (\mathcal{K})=-\frac{\sigma (\mathcal{K})}{2}
\end{equation}

and the following expansion
\begin{equation}
(-t)^{-\frac{\sigma (\mathcal{K})}{2}}\mathcal{P}_{1}(\mathcal{K}%
;a,q,t)=1+(aq^{-1}+t^{-1}a^{-1}q)(t^{2}aq+t^{-1}a^{-1}q^{-1})H_{1}(\mathcal{K%
};a,q,t).
\end{equation}
\end{theorem}

\begin{proof}
By using the skein relation for classical HOMFLY-PT polynomial $P(\mathcal{K}%
;a,q)$, we have%
\begin{equation}
P(\mathcal{K};a,q)=1+(aq^{-1}-a^{-1}q)(aq-a^{-1}q^{-1})f(\mathcal{K};a,q)
\end{equation}

for some function $f(\mathcal{K};a,q)\in
\mathbb{Z}
\lbrack a^{\pm 1},(q-q^{-1})^{2}].$

Now combined with Theorem 1.1, arguments for quasi-alternating knot in \cite%
{MO} and discussion in Sec. 5.2 in \cite{DGR}, we could easily get the
following expansion%
\begin{equation}
(-t)^{-\frac{\sigma (\mathcal{K})}{2}}\mathcal{P}_{1}(\mathcal{K}%
;a,q,t)=1+(aq^{-1}+t^{-1}a^{-1}q)(t^{2}aq+t^{-1}a^{-1}q^{-1})f(\mathcal{K}%
_{1};at,\sqrt{-1}qt^{\frac{1}{2}}).
\end{equation}

with $H_{1}(\mathcal{K};a,q,t)=f(\mathcal{K};at,\sqrt{-1}qt^{\frac{1}{2}%
})\in
\mathbb{Z}
\lbrack a^{\pm 1},q^{\pm 1},t^{\pm 1}]$.
\end{proof}

\bigskip

Torus knots and torus links are studied completely in \cite{DMMSS}. Based on
these highly nontrivial computations, we are able to prove the following
theorem for all torus knots.

\begin{theorem}
For any coprime pair $(m,n)=1$, where $m<n$, cyclotomic expansion conjecture
(Conj. 1.3 or Conj. 2.3) of $N=1$ is true for torus knot $T(m,n)$ and we
have $\alpha (T(m,n))=-(m-1)(n-1)/2$.
\end{theorem}

\begin{remark}
Because \cite{DMMSS} compute the superpolynomial of the mirror of torus knot
$T(m,n)$, we take $\alpha (T(m,n))=(m-1)(n-1)/2$ instead of $-(m-1)(n-1)/2$
in the following proof.
\end{remark}

\begin{proof}
In order to prove that there is an expansion such as%
\begin{eqnarray}
&&(-t)^{(m-1)(n-1)/2}\mathcal{P}_{1}(T(m,km+p);a,q,t) \\
&=&1+H_{1}(T(m,km+p);a,q,t)(aq^{-1}+t^{-1}a^{-1}q)(t^{2}aq+t^{-1}a^{-1}q^{-1})
\notag
\end{eqnarray}
\end{proof}

\begin{theorem}
\begin{proof}
It is sufficient to prove the following two identities,%
\begin{equation}
(-t)^{(m-1)(n-1)/2}\mathcal{P}_{1}(T(m,km+p);a,q,t)|_{a^{2}=-q^{2}t^{-1}}=1
\end{equation}

and%
\begin{equation}
(-t)^{(m-1)(n-1)/2}\mathcal{P}_{1}(T(m,km+p);a,q,t)|_{a^{2}=-q^{-2}t^{-3}}=1
\end{equation}

Compared notations in this paper and in \cite{DMMSS}, there is a notation
change by multiplication of $\frac{a^{(m-1)n}}{q^{(m-1)n}}$.

By setting $n=km+p$, it is easy to know $(m,p)=1$.

The expression of the superpolynomial of triply-graded reduced non-colored
HOMFLY-PT homology of torus knot $T(m,km+p)$ is given by the following%
\begin{equation}
\mathcal{P}_{1}(T(m,km+p);a,q,t)=\frac{a^{(m-1)(km+p)}}{q^{(m-1)(km+p)}}%
\frac{\{\widetilde{t}\}A^{m-1}\widetilde{q}^{(m-1)(km+p)}}{\{A\}\widetilde{t}%
^{m-1}}P(T(m,km+p);a,q,t)
\end{equation}

where $\widetilde{t}=q$, $\widetilde{q}=-qt$, $A=a\sqrt{-t}$ ((48) of \cite%
{DMMSS}) and $\{f(a,q,t)\}=f(a,q,t)-(f(a,q,t))^{-1}$.

The identity $P(T(m,km+p);a,q,t)$ is given by ((4) and (7) of \cite{DMMSS})%
\begin{equation}
P(T(m,km+p);a,q,t)=\underset{|Q|=m}{\sum }\widetilde{q}^{-2(km+p)\nu
(Q^{\prime })/m}\widetilde{t}^{2(km+p)\nu (Q)/m}c_{(1)}^{Q}M_{Q}^{\ast },
\end{equation}

Here $M_{R}^{\ast }$ is given by the following identity ((41) of \cite{DMMSS}%
),
\begin{equation}
M_{R}^{\ast }=\underset{(i,j)\in R}{\prod }\frac{A\widetilde{q}^{j-1}/%
\widetilde{t}^{i-1}-(A\widetilde{q}^{j-1}/\widetilde{t}^{i-1})^{-1}}{%
\widetilde{q}^{k}\widetilde{t}^{l+1}-(\widetilde{q}^{k}\widetilde{t}%
^{l+1})^{-1}},  \label{M*R}
\end{equation}

where $k=R_{i}-j-1$ and $l=R_{j}^{\prime }-i-1$.

Then we immediately get the following expression for $\mathcal{P}%
_{1}(T(m,km+p);a,q,t)$%
\begin{eqnarray}
\mathcal{P}_{1}(T(m,km+p);a,q,t) &=&\frac{a^{(m-1)(km+p)}}{q^{(m-1)(km+p)}}%
\frac{\{\widetilde{t}\}A^{m-1}\widetilde{q}^{(m-1)(km+p)}}{\{A\}\widetilde{t}%
^{m-1}}  \label{P1Expansion} \\
&&\underset{|Q|=m}{\sum }\widetilde{q}^{-2(km+p)\nu (Q^{\prime })/m}%
\widetilde{t}^{2(km+p)\nu (Q)/m}c_{(1)}^{Q}M_{Q}^{\ast },  \notag
\end{eqnarray}

Although it is generally difficult to determine all the coefficients $%
c_{(1)}^{Q}$. In order to prove the expansion formula for $%
(-t)^{(m-1)(n-1)/2}\mathcal{P}_{1}(T(m,km+p);a,q,t)$, actually we don't need
to do so. Many terms of $M_{Q}^{\ast }$ will disappear after they are
evaluated at $a^{2}=-q^{2}t^{-1}$ or $a^{2}=-q^{-2}t^{-3}$.

According to (\ref{M*R}) ((41) of \cite{DMMSS}), we get to know the fact
that $M_{Q}^{\ast }$ contain $A\widetilde{t}^{-1}-A^{-1}\widetilde{t}$ in
the numerator except for $Q=(m)$ and $M_{Q}^{\ast }$ contain $A\widetilde{q}%
-A^{-1}\widetilde{q}^{-1}$ in the numerator except for $Q=(1^{m})$.

In fact, we have the following identities%
\begin{eqnarray}
&&\left( A\widetilde{t}^{-1}-A^{-1}\widetilde{t}\right)
|_{a^{2}=-q^{2}t^{-1}}  \notag \\
&=&\left( A^{-1}(A^{2}q^{-1}-q)\right) |_{a^{2}=-q^{2}t^{-1}}  \notag \\
&=&\left( A^{-1}(-ta^{2}q^{-1}-q)\right) |_{a^{2}=-q^{2}t^{-1}}  \notag \\
&=&0,
\end{eqnarray}

and%
\begin{eqnarray}
&&\left( A\widetilde{q}-A^{-1}\widetilde{q}^{-1}\right)
|_{a^{2}=-q^{-2}t^{-3}}  \notag \\
&=&\left( A^{-1}(-A^{2}qt+q^{-1}t^{-1})\right) |_{a^{2}=-q^{-2}t^{-3}}
\notag \\
&=&\left( A^{-1}(-ta^{2}qt+q^{-1}t^{-1})\right) |_{a^{2}=-q^{-2}t^{-3}}
\notag \\
&=&0,
\end{eqnarray}

In fact, only $M_{(m)}^{\ast }$ survived after they and evaluated at $%
a^{2}=-q^{2}t^{-1}$ and only $M_{(1^{m})}^{\ast }$ survived after they and
evaluated at $a^{2}=-q^{-2}t^{-3}$.

Thus we immediately obtain the following expression for $\mathcal{P}%
_{1}(T(m,km+p);a,q,t)$ evaluated at $a^{2}=-q^{2}t^{-1}$ and $%
a^{2}=-q^{-2}t^{-3}$ from (\ref{P1Expansion}),%
\begin{eqnarray}
&&\left( \mathcal{P}_{1}(T(m,km+p);a,q,t)\right) |_{a^{2}=-q^{2}t^{-1}}
\notag \\
&=&(\frac{a^{(m-1)(km+p)}}{q^{(m-1)(km+p)}}\frac{\{\widetilde{t}\}A^{m-1}%
\widetilde{q}^{(m-1)(km+p)}}{\{A\}\widetilde{t}^{m-1}}  \notag \\
&&\widetilde{q}^{-2(km+p)\nu (1^{m})/m}\widetilde{t}^{2(km+p)\nu
(m)/m}c_{(1)}^{(m)}M_{(m)}^{\ast })|_{a^{2}=-q^{2}t^{-1}}  \notag \\
&=&(\frac{a^{(m-1)(km+p)}}{q^{(m-1)(km+p)}}\frac{\{\widetilde{t}\}A^{m-1}%
\widetilde{q}^{(m-1)(km+p)}}{\{A\}\widetilde{t}^{m-1}}\widetilde{q}%
^{-(km+p)(m-1)}c_{(1)}^{(m)}M_{(m)}^{\ast })|_{a^{2}=-q^{2}t^{-1}}  \notag \\
&=&(\frac{a^{(m-1)(km+p)}}{q^{(m-1)(km+p)}}\frac{\{\widetilde{t}\}A^{m-1}}{%
\{A\}\widetilde{t}^{m-1}}c_{(1)}^{(m)}M_{(m)}^{\ast })|_{a^{2}=-q^{2}t^{-1}}
\label{1stEvaluation}
\end{eqnarray}

and

\begin{eqnarray}
&&\left( \mathcal{P}_{1}(T(m,km+p);a,q,t)\right) |_{a^{2}=-q^{-2}t^{-3}}
\notag \\
&=&(\frac{a^{(m-1)(km+p)}}{q^{(m-1)(km+p)}}\frac{\{\widetilde{t}\}A^{m-1}%
\widetilde{q}^{(m-1)(km+p)}}{\{A\}\widetilde{t}^{m-1}}  \notag \\
&&\widetilde{q}^{-2(km+p)\nu (m)/m}\widetilde{t}^{2(km+p)\nu
(1^{m})/m}c_{(1)}^{(1^{m})}M_{(1^{m})}^{\ast })|_{a^{2}=-q^{-2}t^{-3}} \\
&=&(\frac{a^{(m-1)(km+p)}}{q^{(m-1)(km+p)}}\frac{\{\widetilde{t}\}A^{m-1}%
\widetilde{q}^{(m-1)(km+p)}}{\{A\}\widetilde{t}^{m-1}}\widetilde{t}%
^{(km+p)(m-1)}c_{(1)}^{(1^{m})}M_{(1^{m})}^{\ast })|_{a^{2}=-q^{-2}t^{-3}}
\notag \\
&=&(\frac{a^{(m-1)(km+p)}}{q^{(m-1)(km+p)}}\frac{\{\widetilde{t}\}A^{m-1}%
\widetilde{q}^{(m-1)(km+p)}}{\{A\}}\widetilde{t}%
^{(km+p-1)(m-1)}c_{(1)}^{(1^{m})}M_{(1^{m})}^{\ast })|_{a^{2}=-q^{-2}t^{-3}},
\label{2ndEvaluation}
\end{eqnarray}

where we used $\nu ((m))=0$ and $\nu ((1^{m}))=(m-1)m/2$.

In \cite{DMMSS}, there is a trick to determine the coefficients $c_{(1)}^{Q}$
that one use (1) of \cite{DMMSS} in the following sense,%
\begin{equation}
\mathcal{P}_{1}(T(m,p);a,q,t)=\mathcal{P}_{1}(T(p,m);a,q,t)\text{ with }p<m.
\label{T(m,p)=T(p,m)}
\end{equation}

By induction method, we can assume the following
\begin{equation}
(-t)^{(p-1)(m-1)/2}\mathcal{P}_{1}(T(p,m);a,q,t)|_{a^{2}=-q^{2}t^{-1}}=1
\label{1stT(p,m)}
\end{equation}

and%
\begin{equation}
(-t)^{(p-1)(m-1)/2}\mathcal{P}_{1}(T(p,m);a,q,t)|_{a^{2}=-q^{-2}t^{-3}}=1
\label{2ndT(p,m)}
\end{equation}

We need to prove the following%
\begin{equation}
(-t)^{(m-1)(km+p-1)/2}\mathcal{P}%
_{1}(T(m,km+p);a,q,t)|_{a^{2}=-q^{2}t^{-1}}=1  \label{1stT(m,km+p)}
\end{equation}

and%
\begin{equation}
(-t)^{(m-1)(km+p-1)/2}\mathcal{P}%
_{1}(T(m,km+p);a,q,t)|_{a^{2}=-q^{-2}t^{-3}}=1  \label{2ndT(m,km+p)}
\end{equation}

Conbined with (\ref{1stEvaluation}) and (\ref{2ndEvaluation}), we can
immediately check the initial case for $p=1$ as follows%
\begin{eqnarray}
&&(-t)^{(1-1)(m-1)/2}\left( \mathcal{P}_{1}(T(1,m);a,q,t)\right)
|_{a^{2}=-q^{2}t^{-1}}  \notag \\
&=&\left( \frac{\{\widetilde{t}\}}{\{A\}}c_{(1)}^{(1)}M_{(1)}^{\ast }\right)
|_{a^{2}=-q^{2}t^{-1}}  \notag \\
&=&\left( \frac{\{\widetilde{t}\}}{\{A\}}\frac{A-A^{-1}}{\widetilde{t}-%
\widetilde{t}^{-1}}\right) |_{a^{2}=-q^{2}t^{-1}}  \notag \\
&=&1
\end{eqnarray}

and%
\begin{eqnarray}
&&(-t)^{(1-1)(m-1)/2}\left( \mathcal{P}_{1}(T(1,m);a,q,t)\right)
|_{a^{2}=-q^{-2}t^{-3}}  \notag \\
&=&\left( \frac{\{\widetilde{t}\}}{\{A\}}c_{(1)}^{(1)}M_{(1)}^{\ast }\right)
|_{a^{2}=-q^{-2}t^{-3}}  \notag \\
&=&\left( \frac{\{\widetilde{t}\}}{\{A\}}\frac{A-A^{-1}}{\widetilde{t}-%
\widetilde{t}^{-1}}\right) |_{a^{2}=-q^{-2}t^{-3}}  \notag \\
&=&1
\end{eqnarray}

From (\ref{1stEvaluation}), (\ref{T(m,p)=T(p,m)}) and (\ref{1stT(p,m)}), we
have%
\begin{eqnarray*}
&&(-t)^{(m-1)(km+p-1)/2}\mathcal{P}%
_{1}(T(m,km+p);a,q,t)|_{a^{2}=-q^{2}t^{-1}} \\
&=&(-t)^{(m-1)(km+p-1)/2}\frac{a^{(m-1)(km+p)}}{q^{(m-1)(km+p)}}\frac{\{%
\widetilde{t}\}A^{m-1}}{\{A\}\widetilde{t}^{m-1}}c_{(1)}^{(m)}M_{(m)}^{\ast
}|_{a^{2}=-q^{2}t^{-1}} \\
&=&\left( (-t)^{(m-1)km/2}\frac{a^{(m-1)km}}{q^{(m-1)km}}\right)
|_{a^{2}=-q^{2}t^{-1}}\left( (-t)^{(m-1)(p-1)/2}\mathcal{P}%
_{1}(T(m,p);a,q,t)|_{a^{2}=-q^{2}t^{-1}}\right)  \\
&=&1\cdot \left( (-t)^{(m-1)(p-1)/2}\mathcal{P}%
_{1}(T(p,m);a,q,t)|_{a^{2}=-q^{2}t^{-1}}\right)  \\
&=&1,
\end{eqnarray*}

which is just (\ref{1stT(m,km+p)}).

Similarly, from (\ref{2ndEvaluation}), (\ref{T(m,p)=T(p,m)}) and (\ref%
{2ndT(p,m)}), we have%
\begin{eqnarray*}
&&(-t)^{(m-1)(km+p-1)/2}\mathcal{P}%
_{1}(T(m,km+p);a,q,t)|_{a^{2}=-q^{-2}t^{-3}} \\
&=&(-t)^{(m-1)(km+p-1)/2}\frac{a^{(m-1)(km+p)}}{q^{(m-1)(km+p)}}\frac{\{%
\widetilde{t}\}A^{m-1}\widetilde{q}^{(m-1)(km+p)}}{\{A\}}\widetilde{t}%
^{(km+p-1)(m-1)}c_{(1)}^{(1^{m})}M_{(1^{m})}^{\ast }|_{a^{2}=-q^{-2}t^{-3}}
\\
&=&\left( (-t)^{(m-1)km/2}\frac{a^{(m-1)km}}{q^{(m-1)km}}\widetilde{q}%
^{(m-1)km}\widetilde{t}^{km(m-1)}\right) |_{a^{2}=-q^{-2}t^{-3}} \\
&&\left( (-t)^{(m-1)(p-1)/2}\mathcal{P}%
_{1}(T(m,p);a,q,t)|_{a^{2}=-q^{-2}t^{-3}}\right)  \\
&=&\left( (-t)^{(m-1)km/2}\frac{a^{(m-1)km}}{q^{(m-1)km}}%
(-qt)^{(m-1)km}q^{km(m-1)}\right) |_{a^{2}=-q^{-2}t^{-3}} \\
&&\left( (-t)^{(m-1)(p-1)/2}\mathcal{P}%
_{1}(T(p,m);a,q,t)|_{a^{2}=-q^{-2}t^{-3}}\right)  \\
&=&\left( (-t)^{3(m-1)km/2}(aq)^{(m-1)km}\right) |_{a^{2}=-q^{-2}t^{-3}} \\
&=&1,
\end{eqnarray*}

which is just (\ref{2ndT(m,km+p)}).

Thus we complete our proof.
\end{proof}
\end{theorem}

Now we are considering a problem relating to the sliceness of a knot.

\begin{definition}
The smooth 4-ball genus $g_{4}(\mathcal{K})$ of a knot $\mathcal{K}$ is the
minimum genus of a surface smoothly embedded in the 4-ball $B^{4}$ with
boundary the knot. In particular, a knot $\mathcal{K}\subset S^{3}$ is
called smoothly slice if $g_{4}(\mathcal{K})=0$.
\end{definition}

\begin{remark}
The invariant $\alpha (T(m,n))=-(m-1)(n-1)/2$ suggest a very close relation
between the above theorem and the following Milnor Conjecture, which was
first proved by P. B. Kronheimer and T. S. Mrowka in \cite{KM}
\end{remark}

\begin{conjecture}[Milnor]
The smooth 4-ball genus for torus knot $T(m,n)$ is $(m-1)(n-1)/2$.
\end{conjecture}

Rasmussen \cite{Ras2} introduced a knot concordant invariant $s(\mathcal{K})$%
,\ which is a lower bound for the smooth 4-ball genus for knots in the
following sense.

\begin{theorem}[Rasmussen]
For any knot $\mathcal{K}\subset S^{3}$, we have the following relation%
\begin{equation}
|s(\mathcal{K})|\leq 2g_{4}(\mathcal{K}).
\end{equation}
\end{theorem}

In addition, Rasmussen again proved Milnor Conjecture by a purely
combinatorial method in \cite{Ras2}.

Based on all the above results shown in table or proved via theorem, we are
able to propose the following conjecture

\begin{conjecture}
The invariant $\alpha (\mathcal{K})$ (determined by cyclotomic expansion
conjecture (Conj 1.3 or Conj. 2.3) for $N=1$) is a lower bound for smooth
4-ball genus $g_{4}(\mathcal{K})$, i.e.%
\begin{equation}
|\alpha (\mathcal{K})|\leq g_{4}(\mathcal{K}).
\end{equation}
\end{conjecture}

\begin{remark}
For many knots we tested, it is identical to the Ozsv\'{a}th-Szab\'{o}'s $%
\tau $ invariant and Rasmussen's $s$ invariant.
\end{remark}

\section{Volume conjecture for SU(n) specialized superpolynomial of
HOMFLY-PT homology}

First we present certain motivation to propose our Volume Conjecture for
superpolynomials assocaited to triply-graded reduced colored HOMFLY
homologies.

From (1.7), we have the following expression for figure eight knot $4_{1}$,

\begin{equation}
\mathcal{P}_{N-1}(4_{1};a,q,t)=1+\underset{k=1}{\overset{N-1}{\sum }}%
\underset{i=1}{\overset{k}{\prod }}\left( \frac{\{N-i\}}{\{i\}}%
A_{i-2}(a,q,t)B_{N-2+i}(a,q,t)\right) .
\end{equation}

where $A_{i}(a,q,t)=aq^{i}+t^{-1}a^{-1}q^{-i}$, $%
B_{i}(a,q,t)=t^{2}aq^{i}+t^{-1}a^{-1}q^{-i}$ and $\{p\}=q^{p}-q^{-p}$.

\bigskip

The idea of "Gap" in \cite{CLZ} plays an important role in proposing Volume
Conjectures. The middle terms in the cyclotomic expansion of colored $SU(n)$
invariants of figure eight knot is $\{N\}$ and $\{N+n\}$, thus "Gaps" are $%
N+1$, $N+2$,..., $N+n-1$. We choose these "Gaps" as our roots of unity.

Conjecture presented in \cite{CLZ} is the following

\begin{conjecture}[Volume Conjecture for colored $SU(n)$ invariants
\protect\cite{CLZ}]
For any hyperbolic knot $\mathcal{K}$, we have
\begin{equation*}
2\pi \underset{N\rightarrow \infty }{\lim }\frac{\log J_{N}^{SU(n)}(\mathcal{%
K};e^{\frac{\pi \sqrt{-1}}{N+a}})}{N+1}=Vol(S^{3}\backslash \mathcal{K})+%
\sqrt{-1}CS(S^{3} \backslash \mathcal{K})  \text{ }(mod\sqrt{-1}\pi ^{2}%
\mathbb{Z}
),
\end{equation*}

where $a=1,2,...,n-1$.
\end{conjecture}

Now we apply the same motivation of "Gaps" here, which seems a little bit
more complicated. Because "Gaps" in cyclotomic expansion of colored $SU(n)$
invariants of figure eight knot are essentially certain equation with only $%
q $ involved; while "Gaps" in cyclotomic expansion of $SU(n)$ specialized
superpolynomial of colored HOMFLY-PT homology are equations of both $q$ and $t$
involved.

\bigskip

By looking at middle terms $A_{N-3}(q^{n},q,t)=q^{N+n-3}+t^{-1}q^{-(N+n-3)}$
and $B_{N-1}(q^{n},q,t)=(-t)^{2}q^{N-1+n}+t^{-1}q^{-(N+n-1)}$, we get to
know the possible "Gaps" is the following equation

\begin{equation*}
(-t)q^{N+n-2}+t^{-1}q^{-(N+n-2)}=0
\end{equation*}

By solving equation%
\begin{equation*}
tq^{N+n-2}=t^{-1}q^{-(N+n-2)},
\end{equation*}

We take one solution
\begin{equation*}
t=q^{-(N+n-2)}
\end{equation*}

Thus we obtain the following expression for $A_{i}$ and $B_{i}$,%
\begin{eqnarray*}
A_{i}(q^{n},q,q^{-(N+n-2)}) &=&q^{i+n}+q^{N-i-2} \\
B_{i}(q^{n},q,q^{-(N+n-2)}) &=&q^{-2N-n+i+4}+q^{N-i-2}
\end{eqnarray*}

Then we express the $A_{i}$ and $B_{i}$ at roots of unity $q=e^{\frac{\pi
\sqrt{-1}}{N-1+b}}=e^{\frac{\pi \sqrt{-1}}{N+\widetilde{b}}}$, where $%
\widetilde{b}=b-1$, in the following way%
\begin{eqnarray*}
A_{i}(q^{n},q,q^{-(N+n-2)}) &=&q^{\frac{n-\widetilde{b}-2}{2}}(q^{\frac{%
\widetilde{b}+2i-n+2}{2}}-q^{-\frac{\widetilde{b}+2i-n+2}{2}}) \\
B_{i}(q^{n},q,q^{-(N+n-2)}) &=&q^{\frac{\widetilde{b}-n+2}{2}}(q^{\frac{n-2i-%
\widetilde{b}-2}{2}}-q^{-\frac{n-2i-\widetilde{b}-2}{2}})
\end{eqnarray*}

Now we are able to write down the $SU(n)$ specialized superpolynomial $%
\mathcal{P}_{N-1}(4_{1};q^{n},q,q^{-(N+n-2)})$ at roots of unity $q=e^{\frac{%
\pi \sqrt{-1}}{N-1+b}}$,%
\begin{equation}
\mathcal{P}_{N-1}(4_{1};q^{n},q,q^{-(N+n-2)})=1+\underset{j=1}{\overset{N-1}{%
\sum }}g(N,j)\text{,}
\end{equation}

where $g(N,j)=\underset{k=1}{\overset{j}{\prod }}f(N,k)$ and $f(N,k)=4\frac{%
\sin \frac{(\frac{n-2+\widetilde{b}}{2}+k)\pi }{N+\widetilde{b}}}{\sin \frac{%
k\pi }{N+\widetilde{b}}}\sin \frac{(k+\widetilde{b})\pi }{N+\widetilde{b}}%
\sin \frac{(-\frac{n-2-\widetilde{b}}{2}+k)\pi }{N+\widetilde{b}}$.

\begin{remark}
For $b=n-1$, i.e. $\widetilde{b}=n-2$ and $t=q^{-(N+n-2)}=e^{-\frac{N+n-2}{N+%
\widetilde{b}}}=-1$, which is the decategorified case,, we have $%
f(N,k)=\left( 2\sin \frac{(n-2+k)\pi }{N+n-2}\right) ^{2}$. This corresponds
to the (3.25) of \cite{CLZ}$.$
\end{remark}

Now we propose volume conjecture for $SU(n)$ specialized superpolynomials of
HOMFLY-PT homology as follows

\begin{conjecture}[Volume Conjecture for $SU(n)$ specialized superpolynomial]

For any hyperbolic knot $\mathcal{K}$, we have%
\begin{equation*}
2\pi \underset{N\rightarrow \infty }{\lim }\frac{\log \mathcal{P}_{N}(%
\mathcal{K};q^{n},q,q^{-(N+n-1)})|_{q=e^{\frac{\pi \sqrt{-1}}{N+b}}}}{N+1}%
=Vol(S^{3}\backslash \mathcal{K})+\sqrt{-1}CS(S^{3}\backslash \mathcal{K}) \text{ }(mod\sqrt{-1}\pi ^{2}%
\mathbb{Z}
),
\end{equation*}

where $b\geq 1$ and $\frac{n-1-b}{2}$ is not a positive integer.
\end{conjecture}

\begin{remark}
The condition of this conjecture is
much more relaxed than condition of former Volume conjectures, because here $b$ can be
any larger integers. For example, original volume conjecture is only valid for $%
n=2$ and $b=1$, but this volume conjecture is valid for all positive integer $b$
with $n=2$.
\end{remark}

\begin{remark}
It will be interesting to know the relationship of between this volume conjecture and
the one proposed by Fuji, Gukov and Sulkowski in \cite{FGS1}, where they used categorified A-polynomials
of knots.
\end{remark}

We prove this volume conjecture for the case of figure eight knot $4_{1}$.

\begin{theorem}
The above volume conjecture is valid for figure eight knot $4_{1}$.
\end{theorem}

\begin{proof}
The $SU(n)$ specialized superpolynomial $\mathcal{P}%
_{N-1}(4_{1};q^{n},q,q^{-(N+n-2)})$ at roots of unity $q=e^{\frac{\pi \sqrt{%
-1}}{N-1+b}}$ is given by%
\begin{equation}
\mathcal{P}_{N-1}(4_{1};q^{n},q,q^{-(N+n-2)})=1+\underset{j=1}{\overset{N-1}{%
\sum }}g(N,j)\text{,}
\end{equation}%
where $g(N,j)=\underset{k=1}{\overset{j}{\prod }}f(N,k)$, $f(N,k)=4\frac{%
\sin \frac{(\frac{n-2+\widetilde{b}}{2}+k)\pi }{N+\widetilde{b}}}{\sin \frac{%
k\pi }{N+\widetilde{b}}}\sin \frac{(k+\widetilde{b})\pi }{N+\widetilde{b}}%
\sin \frac{(-\frac{n-2-\widetilde{b}}{2}+k)\pi }{N+\widetilde{b}}$ and $%
\widetilde{b}=b-1$.

Condition that $\frac{n-1-b}{2}$ is not a positive integer assures that $%
\sin \frac{(-\frac{n-2-\widetilde{b}}{2}+k)\pi }{N+\widetilde{b}}\neq 0$ for
$1\leq k\leq N-1$.

There is an fact that $f(N,k)$ can be negative only for very small $k$.

Thus when we do the estimation for $g(N,k)$, the sign of $g(N,k)$ is not
changed for $k>\frac{n-2-\widetilde{b}}{2}$.

Similar to the proof in \cite{CLZ}, our task is to search for $k_{m}$ such
that $|g(N,k)|$ reaches its maximum value.

We claim the following inequalities (Similar to Lemma 3.7 of \cite{CLZ}),
\begin{equation*}
\left\lfloor \frac{5}{6}(N+\widetilde{b})-\frac{3(n-2)+7\widetilde{b}}{4}%
\right\rfloor \leq k_{m}\leq \left\lfloor \frac{5}{6}(N+\widetilde{b})+\frac{%
n-2-\widetilde{b}}{2}\right\rfloor ,
\end{equation*}

where the floor function $\left\lfloor x\right\rfloor $ denotes the greatest
integer that is less than or equal to $x$.

The upper bound of $k_{m}$ is clear, in fact if $k_{m}\geq \frac{5}{6}(N+%
\widetilde{b})+\frac{n-2-\widetilde{b}}{2}$, then $f(N,k_{m})<1$

We need to estimate a lower bound of $k_{m}$, we can assume $\frac{1}{2}\leq
k_{m}\leq \frac{11}{12}$

\begin{equation*}
\frac{\sin \frac{(\frac{n-2+\widetilde{b}}{2}+k)\pi }{N+\widetilde{b}}}{\sin
\frac{k\pi }{N+\widetilde{b}}}=\sin \frac{\frac{n-2+\widetilde{b}}{2}\pi }{N+%
\widetilde{b}}\cot \frac{k\pi }{N+\widetilde{b}}+\cos \frac{\frac{n-2+%
\widetilde{b}}{2}\pi }{N+\widetilde{b}}
\end{equation*}

Set $\frac{n-2+\widetilde{b}}{2(N+\widetilde{b})}\pi =\alpha $ for $\frac{1}{%
2}\leq k\leq \frac{11}{12}$, we have $\frac{\sin \frac{(\frac{n-2+\widetilde{%
b}}{2}+k)\pi }{N+\widetilde{b}}}{\sin \frac{k\pi }{N+\widetilde{b}}}\geq 1-%
\frac{1}{2}\alpha ^{2}-\cot \frac{\pi }{12}\alpha ,$

where we used the inequality: $\sin \alpha <\alpha $ and $\cos \alpha >1-%
\frac{1}{2}\alpha ^{2}$ for small $\alpha >0$.

We also have%
\begin{eqnarray*}
4\sin \frac{(k+\widetilde{b})\pi }{N+\widetilde{b}}\sin \frac{(-\frac{n-2-%
\widetilde{b}}{2}+k)\pi }{N+\widetilde{b}} &\geq &4\sin ^{2}\frac{(k+%
\widetilde{b})\pi }{N+\widetilde{b}} \\
&=&4\sin ^{2}(\frac{5\pi }{6}-\beta ) \\
&=&1+2\sqrt{3}\sin \beta +2\sin ^{2}\beta ,
\end{eqnarray*}

where $\beta =\frac{5\pi }{6}-\frac{(k+\widetilde{b})\pi }{N+\widetilde{b}}.$

Thus we have%
\begin{eqnarray*}
f(N,k) &=&4\frac{\sin \frac{(\frac{n-2+\widetilde{b}}{2}+k)\pi }{N+%
\widetilde{b}}}{\sin \frac{k\pi }{N+\widetilde{b}}}\sin \frac{(k+\widetilde{b%
})\pi }{N+\widetilde{b}}\sin \frac{(-\frac{n-2-\widetilde{b}}{2}+k)\pi }{N+%
\widetilde{b}} \\
&\geq &(1-\frac{1}{2}\alpha ^{2}-\cot \frac{\pi }{12}\alpha )(1+2\sqrt{3}%
\sin \beta +2\sin ^{2}\beta ) \\
&=&1+2\sqrt{3}\beta -(2+\sqrt{3})\alpha +O(\alpha ^{2})+O(\beta ^{2})
\end{eqnarray*}

If we let $k_{0}=\frac{5}{6}(N+\widetilde{b})-\frac{3(n-2)+7\widetilde{b}}{4}
$ Then $\beta =\frac{3}{2}\alpha $, we have $f(N,k)>1$.

By a similar argument in Lemma 3.8 of \cite{CLZ} (corresponding to $s=1$
case there) and remembering that $\sin \frac{(-\frac{n-2-\widetilde{b}}{2}%
+k)\pi }{N+\widetilde{b}}$ could take negative values for small integer $k$,
we have%
\begin{equation*}
|g(N,k_{m})|\leq |\mathcal{P}_{N-1}(4_{1};q^{n},q,q^{-(N+n-2)})|\leq
N|g(N,k_{m})|
\end{equation*}
for sufficient large $N$.

By the method in the proof of Lemma 3.5 and argument in Proposition 3.10 in
\cite{CLZ}, we could finish the proof.%
\begin{eqnarray*}
&&\underset{N\rightarrow \infty }{\lim }\frac{2\pi \log \mathcal{P}%
_{N-1}(4_{1};q^{n},q,q^{-(N+n-2)})}{N} \\
&=&2\pi \frac{5}{6}\log 4+2\pi \frac{2}{\pi }\int\limits_{0}^{\frac{5\pi }{6}%
}\log |\sin (t)|dt \\
&=&4\int\limits_{0}^{\frac{5\pi }{6}}\log |2\sin (t)|dt \\
&=&6\Lambda (\pi /3) \\
&=&Vol(S^{3}\backslash 4_{1})
\end{eqnarray*}
\end{proof}

\section{Superpolynomials of colored Kauffman homology}

In this section, we study cyclotomic expansion for superpolynomial $\mathcal{%
F}_{N}(\mathcal{K};a,q,t)$ of triply-graded reduced colored Kauffman
homology formulated by S. Gukov and J. Walcher in \cite{GW}. We obtain the
similar expansion conjecture.

\begin{conjecture}
For any knot $\mathcal{K}$, there exists an integer valued invariant $\beta (\mathcal{K})\in
\mathbb{Z}
$, s.t. the superpolynomial $\mathcal{F}_{N}(\mathcal{K};a,q,t)$
of triply-graded reduced colored Kauffman homology of a knot $\mathcal{K}$
has the following cyclotomic expansion formula%
\begin{equation}
(-t)^{N\beta (\mathcal{K})}\mathcal{F}_{N}(\mathcal{K};a^{2},q^{2},t)=1+%
\underset{k=1}{\overset{N}{\sum }}F_{k}(\mathcal{K};a,q,t)\left( A_{-1}(a,q,t)%
\underset{i=1}{\overset{k}{\prod }}\left( \frac{\{2(N+1-i)\}}{\{2i\}}%
B_{N+i-2}(a^{2},q^{2},t)\right) \right)
\end{equation}

with coefficient functions $F_{k}(\mathcal{K};a,q,t)\in
\mathbb{Z}
\lbrack a^{\pm 1},q^{\pm 1},t^{\pm 1}]$, where $%
A_{i}(a,q,t)=aq^{i}+t^{-1}a^{-1}q^{-i}$, $%
B_{i}(a,q,t)=t^{2}aq^{i}+t^{-1}a^{-1}q^{-i}$ and $\{p\}=q^{p}-q^{-p}$.

In particular, one further have $\frac{F_{1}(\mathcal{K};a,q,t)}{%
taq+t^{-1}a^{-1}q^{-1}}\in
\mathbb{Z}
\lbrack a^{\pm 1},q^{\pm 1},t^{\pm 1}].$
\end{conjecture}

\begin{remark}
The above Conjecture-Definition for invariant $\beta (\mathcal{K})$ should be
understood in this way. If the above conjecture of a knot $\mathcal{K}$ is true for $N=1$,
then $\beta (\mathcal{K})$ is defined. The next level of the conjecture is for $N\geq
2$ by using the same $\beta (\mathcal{K})$. In this way, $\beta (\mathcal{K})$ is defined even
though the conjecture is only true for $N=1$.
\end{remark}

\begin{remark}
$F_{k}(\mathcal{K};a,q,t)$ is independent of $N$, which only depends
on knot $\mathcal{K}$ and integer $k$.
\end{remark}

\begin{remark}
One can also make the conjecture for $\mathcal{F}_{N}(\mathcal{K};a,q,t)$
instead of $\mathcal{F}_{N}(\mathcal{K};a^{2},q^{2},t)$, but one will get a
factor $a+t^{-1}q$ instead of $A_{1}(a,q,t)=aq^{-1}+t^{-1}a^{-1}q$, which is
a symmetric form by setting $t=-1$.
\end{remark}

For instance, we have the following expansion for $N=1$ and $2$.%
\begin{equation}
(-t)^{\beta (\mathcal{K})}\mathcal{F}_{1}(\mathcal{K};a,q,t)=1+F_{1}(%
\mathcal{K};a,q,t)(aq^{-1}+t^{-1}a^{-1}q^{1})(t^{2}a^{2}+t^{-1}a^{-2})
\end{equation}

and

\begin{eqnarray}
(-t)^{2\beta (\mathcal{K})}\mathcal{F}_{2}(\mathcal{K};a,q,t) &=&1+F_{1}(%
\mathcal{K}%
;a,q,t)(aq^{-1}+t^{-1}a^{-1}q^{1})(q^{2}+q^{-2})(t^{2}a^{2}q^{2}+t^{-1}a^{-2}q^{-2})
\notag \\
&&+F_{2}(\mathcal{K}%
;a,q,t)(aq^{-1}+t^{-1}a^{-1}q^{1})(t^{2}a^{2}q^{2}+t^{-1}a^{-2}q^{-2})(t^{2}a^{2}q^{4}+t^{-1}a^{-2}q^{-4})
\end{eqnarray}

\bigskip

Now we list a table of these cyclotomic expansion coefficients of
superpolynomials for colored Kauffman Homology with small crossing numbers,
where we used tables from pp40 in \cite{GW}.

\bigskip

$%
\begin{array}{llll}
\mathcal{K} & \sigma (\mathcal{K}) & \beta (\mathcal{K}) & K_{1}(\mathcal{K}%
,a,q,t)/(aqt+a^{-1}q^{-1}t^{-1}) \\
3_{1} & -2 & 2 & -a^{4}t^{3}+a^{6}q^{\underline{2}}t^{4}+a^{6}q^{2}t^{5} \\
&  &  &  \\
4_{1} & 0 & 0 & q^{\underline{4}}t^{\underline{1}}+1+q^{4}t^{1} \\
&  &  &  \\
5_{1} & -4 & 4 &
\begin{array}{l}
-a^{4}t^{3}+a^{6}q^{\underline{2}}t^{4}+a^{6}q^{2}t^{5}-a^{8}q^{\underline{4}%
}t^{5}-a^{8}q^{4}t^{7} \\
+a^{10}q^{\underline{6}}t^{6}+a^{10}q^{6}t^{9}+a^{14}q^{\underline{2}%
}t^{10}+a^{14}q^{2}t^{11}%
\end{array}
\\
&  &  &  \\
5_{2} & -2 & 2 & -a^{4}t^{3}+a^{6}q^{\underline{2}%
}t^{4}+a^{6}q^{2}t^{5}+a^{8}t^{6}+a^{10}q^{\underline{6}}t^{6}+a^{10}q^{%
\underline{2}}t^{7}+a^{10}q^{2}t^{8}+a^{10}q^{6}t^{9} \\
&  &  &  \\
6_{1} & 0 & 0 &
\begin{array}{l}
q^{\underline{4}}t^{\underline{1}}+1+q^{4}t^{1}+a^{2}q^{\underline{2}}t^{%
\underline{1}}+a^{2}q^{2}t^{2} \\
+a^{4}q^{\underline{8}}t^{1}+a^{4}q^{\underline{4}%
}t^{2}+a^{4}t^{3}+a^{4}q^{4}t^{4}+a^{4}q^{8}t^{5}%
\end{array}
\\
&  &  &  \\
6_{2} & -2 & 2 &
\begin{array}{l}
a^{4}q^{\underline{8}}t^{1}+a^{4}q^{\underline{4}%
}t^{2}+a^{4}t^{3}+a^{4}q^{4}t^{4}+a^{4}q^{8}t^{5}+a^{6}q^{\underline{6}}t^{3}
\\
+2a^{6}q^{\underline{2}}t^{4}+2a^{6}q^{2}t^{5}+a^{6}q^{6}t^{6}+a^{8}q^{%
\underline{4}}t^{5}+2a^{8}t^{6}+a^{8}q^{4}t^{7}%
\end{array}
\\
&  &  &  \\
6_{3} & 0 & 0 &
\begin{array}{l}
a^{\underline{2}}q^{\underline{6}}t^{\underline{3}}+2a^{\underline{2}}q^{%
\underline{2}}t^{\underline{2}}+2a^{\underline{2}}q^{2}t^{\underline{1}}+a^{%
\underline{2}}q^{6}+q^{\underline{8}}t^{\underline{2}}+2q^{\underline{4}}t^{%
\underline{1}} \\
+3+2q^{4}t^{1}+q^{8}t^{2}+a^{2}q^{\underline{6}}+2a^{2}q^{\underline{2}%
}t+2a^{2}q^{2}t^{2}+a^{2}q^{6}t^{3}%
\end{array}
\\
&  &  &  \\
8_{19} & 6 & -6 &
\begin{array}{l}
a^{\underline{22}}q^{\underline{2}}t^{\underline{18}}+a^{\underline{22}}q^{%
\underline{2}}t^{\underline{17}}+a^{\underline{22}}q^{2}t^{\underline{17}%
}+a^{\underline{22}}q^{2}t^{\underline{16}}+a^{\underline{18}}q^{\underline{6%
}}t^{\underline{16}}+a^{\underline{18}}q^{\underline{6}}t^{\underline{15}}
\\
+a^{\underline{18}}q^{6}t^{\underline{13}}+a^{\underline{18}}q^{6}t^{%
\underline{12}}+a^{\underline{16}}t^{\underline{13}}+a^{\underline{14}}q^{%
\underline{10}}t^{\underline{13}}+a^{\underline{14}}q^{\underline{2}}t^{%
\underline{11}} \\
+a^{\underline{14}}q^{10}t^{\underline{8}}+a^{\underline{14}}q^{2}t^{%
\underline{10}}-a^{\underline{12}}q^{\underline{8}}t^{\underline{11}}-a^{%
\underline{12}}t^{\underline{9}}-a^{\underline{12}}q^{8}t^{\underline{7}}+a^{%
\underline{10}}q^{\underline{6}}t^{\underline{9}}+a^{\underline{10}}q^{6}t^{%
\underline{6}} \\
-a^{\underline{8}}q^{\underline{4}}t^{\underline{7}}-a^{\underline{8}%
}q^{4}t^{\underline{5}}+a^{\underline{6}}q^{\underline{2}}t^{\underline{5}%
}+a^{\underline{6}}q^{2}t^{\underline{4}}-a^{\underline{4}}t^{\underline{3}}%
\end{array}
\\
&  &  &  \\
9_{42} & 2 & 0 &
\begin{array}{l}
a^{\underline{2}}q^{\underline{6}}t^{\underline{5}}+a^{\underline{2}}q^{%
\underline{6}}t^{\underline{4}}+a^{\underline{2}}q^{\underline{2}}t^{%
\underline{4}}+a^{\underline{2}}q^{\underline{2}}t^{\underline{3}}+a^{%
\underline{2}}q^{2}t^{\underline{3}}+a^{\underline{2}}q^{2}t^{\underline{2}%
}+a^{\underline{2}}q^{6}t^{\underline{2}} \\
+a^{\underline{2}}q^{6}t^{\underline{1}}+q^{\underline{12}}t^{\underline{5}%
}+q^{\underline{8}}t^{\underline{4}}+q^{\underline{4}}t^{\underline{3}%
}+q^{4}t^{\underline{1}}+q^{8}+q^{12}t^{1}+a^{2}q^{\underline{6}}t^{%
\underline{2}}+a^{2}q^{\underline{6}}t^{\underline{1}} \\
+a^{2}q^{\underline{2}}t^{\underline{1}}+a^{2}q^{\underline{2}}+2a^{2}t^{%
\underline{2}}+2a^{2}t^{\underline{1}%
}+a^{2}q^{2}+a^{2}q^{2}t^{1}+a^{2}q^{6}t^{1}+a^{2}q^{6}t^{2}%
\end{array}%
\end{array}%
$

\bigskip

Now we listed the tables for knot $3_{1}$ and $4_{1}$ with higher
representation involved, where we used data from \cite{NRZ2}. Indeed, we
checked much higher representation, we only listed results for $K_{2}(%
\mathcal{K},a,q,t)$.

\bigskip

$%
\begin{array}{llll}
\mathcal{K} & \sigma (\mathcal{K}) & \beta (\mathcal{K}) & K_{2}(\mathcal{K}%
,a,q,t) \\
3_{1} & -2 & 2 &
\begin{array}{l}
a^{5}q^{1}t^{3}-a^{7}q^{\underline{1}%
}t^{4}-a^{9}q^{1}t^{5}-a^{9}q^{1}t^{6}-a^{9}q^{5}t^{6}-a^{9}q^{5}t^{7}-a^{9}q^{9}t^{7}
\\
+a^{11}q^{\underline{1}%
}t^{6}+a^{13}q^{5}t^{8}+a^{13}q^{9}t^{8}+2a^{13}q^{9}t^{9}+a^{13}q^{13}t^{9}+a^{13}q^{13}t^{10}
\\
+a^{13}q^{17}t^{10}+a^{13}q^{21}t^{11}+a^{15}q^{11}t^{10}+a^{15}q^{15}t^{11}+a^{15}q^{19}t^{11}+a^{15}q^{23}t^{12}%
\end{array}
\\
&  &  &  \\
4_{1} & 0 & 0 &
\begin{array}{l}
a^{\underline{3}}q^{\underline{19}}t^{\underline{5}}+a^{\underline{3}}q^{%
\underline{15}}t^{\underline{4}}+a^{\underline{3}}q^{\underline{11}}t^{%
\underline{4}}+2a^{\underline{3}}q^{\underline{7}}t^{\underline{3}}+a^{%
\underline{3}}q^{\underline{3}}t^{\underline{3}}+a^{\underline{3}}q^{%
\underline{3}}t^{\underline{2}}+a^{\underline{3}}q^{1}t^{\underline{2}} \\
+a^{\underline{3}}q^{5}t^{\underline{1}}+a^{\underline{1}}q^{\underline{17}%
}t^{\underline{4}}+2a^{\underline{1}}q^{\underline{13}}t^{\underline{3}}+2a^{%
\underline{1}}q^{\underline{9}}t^{\underline{3}}+2a^{\underline{1}}q^{%
\underline{9}}t^{\underline{2}}+4a^{\underline{1}}q^{\underline{5}}t^{%
\underline{2}}+a^{\underline{1}}q^{\underline{5}}t^{\underline{1}} \\
+a^{\underline{1}}q^{\underline{1}}t^{\underline{2}}+4a^{\underline{1}}q^{%
\underline{1}}t^{\underline{1}}+3a^{\underline{1}}q^{3}t^{\underline{1}}+2a^{%
\underline{1}}q^{3}+3a^{\underline{1}}q^{7}+a^{\underline{1}}q^{7}t^{1}+a^{%
\underline{1}}q^{11}t^{1} \\
+a^{1}q^{\underline{11}}t^{\underline{2}}+a^{1}q^{\underline{7}}t^{%
\underline{2}}+3a^{1}q^{\underline{7}}t^{\underline{1}}+2a^{1}q^{\underline{3%
}}t^{\underline{1}}+3a^{1}q^{\underline{3}}+4a^{1}q^{1}+a^{1}q^{1}t^{1} \\
+a^{1}q^{5}+4a^{1}q^{5}t^{1}+2a^{1}q^{9}t^{1}+2a^{1}q^{9}t^{2}+2a^{1}q^{13}t^{2}+a^{1}q^{17}t^{3}+a^{3}q^{%
\underline{5}} \\
+a^{3}q^{\underline{1}%
}t^{1}+a^{3}q^{3}t^{1}+a^{3}q^{3}t^{2}+2a^{3}q^{7}t^{2}+a^{3}q^{11}t^{3}+a^{3}q^{15}t^{3}+a^{3}q^{19}t^{4}%
\end{array}%
\end{array}%
$

\begin{remark}
For these examples we tested, we found that $\beta (\mathcal{K})=2\alpha (%
\mathcal{K})$.
\end{remark}

\section{Poincare Polynomial of Heegaard-Floer Knot Homology}

There is a well-known result that Heeggard-Floer homology of an alternative
knot can be determined by a very simple method with only Alexander
polynomials and signature involved. This result was proved by Ozsv\'{a}%
th-Szab\'{o} \cite{OS1}.

\begin{theorem}[Ozsv\'{a}th-Szab\'{o}]
Let K$\subset S^{3}$ be an alternating knot with Alexander-Conway polynomial $%
\Delta _{K}(q)=\underset{s\in
\mathbb{Z}
}{\sum }a_{s}q^{s}$ and signature $\sigma =\sigma (K)$. Then we have%
\begin{equation}
\widehat{HFK}_{i}(K,s)=\left\{
\begin{array}{c}
\mathbb{Z}
^{|a_{s}|} \\
0%
\end{array}%
\right.
\begin{array}{c}
\text{if }i=s+\frac{\sigma }{2} \\
\text{otherwise}%
\end{array}%
\end{equation}
\end{theorem}

It was shown by C. Manolescu and P.S. Ozsv\'{a}th \cite{MO} that
quasi-alternating knots hold the same results. So it is trivial to check for
these invariants.

Thus we only focus on those homological thick knots with small crossing
numbers described by M. Khovanov on pp. 3 in \cite{Kho2} (up to $10$
crossings) and further test $41$ homologically thick knots up to $11$
crossings.

We observe an expansion formula for Poincare polynomial of Heegaard-Floer
knot homology.

For a knot $\mathcal{K}$, there exists an integer valued invariant $\gamma (%
\mathcal{K})\in
\mathbb{Z}
$ of a knot $\mathcal{K}$, s.t. Poincare polynomial $HFK(\mathcal{K};q^{2},t)
$ of Heegaard-Floer knot homology of a knot $\mathcal{K}$ has the following
expansion formula%
\begin{equation}
(-t)^{\gamma (\mathcal{K})}HFK(\mathcal{K};q^{2},t)=1+KF(\mathcal{K}%
;q,t)(q+t^{-1}q^{-1})^{2}  \label{Heegaard-Floer Expansion}
\end{equation}

with coefficient functions $KF(\mathcal{K};q,t)\in
\mathbb{Z}
\lbrack q^{\pm 1},t^{\pm 1}]$.

\bigskip

Similar to the invariant $\tau $ introduced in Heegaard-Floer theory, we
also have the following theorem for quasi-alternating knots.

\begin{theorem}
The above expansion formula (\ref{Heegaard-Floer Expansion}) is true for any quasi-alternating knot $%
\mathcal{K}$. Furthermore, we have
\begin{equation}
\gamma (\mathcal{K})=-\frac{\sigma (\mathcal{K})}{2}
\end{equation}

and the following expansion
\begin{equation}
(-t)^{-\frac{\sigma (\mathcal{K})}{2}}HFK(\mathcal{K}%
;q^{2},t)=1+(q+t^{-1}q^{-1})^{2}KF(\mathcal{K};q,t),
\end{equation}

where $\sigma (\mathcal{K})$ is the signature of a knot $\mathcal{K}$.
\end{theorem}

\begin{proof}
By using the skein relation for classical Alexander polynomial, we have%
\begin{equation}
\Delta _{\mathcal{K}}(q^{2})=1+(q-q^{-1})^{2}f(\mathcal{K};q)
\end{equation}

for some function $f(\mathcal{K};q)\in
\mathbb{Z}
\lbrack (q-q^{-1})^{2}].$

Now combined with Theorem 4.1 and arguments for quasi-alternating knot in
\cite{MO}, we could easily get the following expansion%
\begin{equation}
(-t)^{-\frac{\sigma (\mathcal{K})}{2}}HFK(\mathcal{K}%
;q^{2},t)=1-t(q+t^{-1}q^{-1})^{2}f(\mathcal{K};\sqrt{-1}qt^{\frac{1}{2}}).
\end{equation}

with $KF(\mathcal{K};q,t)=-tf(\mathcal{K};\sqrt{-1}qt^{\frac{1}{2}})\in
\mathbb{Z}
\lbrack q^{\pm 1},t^{\pm 1}]$.\bigskip
\end{proof}

\bigskip

Now we prove some series examples of whitehead doubles, which has particular
interest for topologists.

In \cite{Hed}, M. Hedden obtain the following Heegaard-Floer homology for
the iterated Whitehead doubles of figure eight knot.

\begin{theorem}
Let $4_{1}$\ be the figure eight knot and let $D^{n}$\ denote the n-th
iterated untwisted Whitehead double of $4_{1}$ i.e. $D^{0}=4_{1}$, $%
D^{n}=D_{+}(D^{n-1},0)$, then we have

$\widehat{HFK}_{\ast }(D^{n},i)\cong \left\{
\begin{array}{c}
\overset{n}{\underset{k=0}{\bigoplus }}%
\mathbb{Z}
_{(1-k)}^{2^{n}\binom{n}{k}} \\
\mathbb{Z}
_{(0)}\overset{n}{\underset{k=0}{\bigoplus }}%
\mathbb{Z}
_{(-k)}^{2^{n+1}\binom{n}{k}} \\
\overset{n}{\underset{k=0}{\bigoplus }}%
\mathbb{Z}
_{(-1-k)}^{2^{n}\binom{n}{k}} \\
0%
\end{array}%
\right.
\begin{array}{c}
i=1 \\
\\
i=0 \\
\\
i=-1 \\
\\
otherwise%
\end{array}%
$
\end{theorem}

Thus we are able to write the corresponding Poincare polynomial of $\widehat{%
HFK}_{\ast }(D^{n},i)$ as follows%
\begin{equation}
HFK(D^{n};q^{2},t)=1+2^{n}(1+t^{-1})^{n}(tq^{2}+2+t^{-1}q^{-2}).
\end{equation}

Then we immediately obtain the following theorem, which verify the expansion
conjecture of Poincare polynomial of Heegaard-Floer homology

\begin{theorem}
The expansion formula (\ref{Heegaard-Floer Expansion}) is valid for n-th iterated untwisted Whitehead
double $D^{n}$ of $4_{1}$. In fact, the invariant $\gamma (D^{n})=0$ and
Poincare polynomial have the following expansion with coefficient $%
KF(D^{n};q,t)=2^{n}t(1+t^{-1})^{n}$.%
\begin{equation}
(-t)^{\gamma
(D^{n})}HFK(D^{n};q^{2},t)=1+2^{n}t(1+t^{-1})^{n}(q+t^{-1}q^{-1})^{2}.
\end{equation}
\end{theorem}

Now we look at another example of Whitehead double.

Followed the idea from \cite{Hed}, K. Park \cite{Par} explicitly obtain the
following Heegaard-Floer homology for the untwisted whitehead double of
torus knot $T(2,2m+1)$.

\begin{theorem}
Let $D(T(2,2m+1))$\ denote the untwisted Whitehead double of torus knot $%
T(2,2m+1)$, then we have

$\widehat{HFK}_{\ast }(D(T(2,2m+1)),i)\cong \left\{
\begin{array}{c}
\mathbb{Z}
_{(0)}^{2m}\bigoplus
\mathbb{Z}
_{(-1)}^{2}\bigoplus
\mathbb{Z}
_{(-3)}^{2}\bigoplus \cdot \cdot \cdot \bigoplus
\mathbb{Z}
_{(-2m+1)}^{2} \\
\mathbb{Z}
_{(-1)}^{4m-1}\bigoplus
\mathbb{Z}
_{(-2)}^{4}\bigoplus
\mathbb{Z}
_{(-4)}^{4}\bigoplus \cdot \cdot \cdot \bigoplus
\mathbb{Z}
_{(-2m)}^{4} \\
\mathbb{Z}
_{(-2)}^{2m}\bigoplus
\mathbb{Z}
_{(-3)}^{2}\bigoplus
\mathbb{Z}
_{(-5)}^{2}\bigoplus \cdot \cdot \cdot \bigoplus
\mathbb{Z}
_{(-2m-1)}^{2} \\
0%
\end{array}%
\right.
\begin{array}{c}
i=1 \\
i=0 \\
i=-1 \\
otherwise%
\end{array}%
$
\end{theorem}

Thus we are able to write the corresponding Poincare polynomial of $\widehat{%
HFK}_{\ast }(D(T(2,2m+1)),i)$ as follows%
\begin{equation}
HFK(D(T(2,2m+1));q^{2},t)=2mq^{2}+(4m-1)t^{-1}+2mt^{-2}q^{-2}+(2t^{-1}q^{2}+4t^{-2}+2t^{-3}q^{-2})%
\frac{1-t^{-2m}}{1-t^{-2}}.
\end{equation}

We have the following expansion%
\begin{equation}
(-t)^{1}HFK(D(T(2,2m+1));q^{2},t)-1=-2(q+t^{-1}q^{-1})^{2}(mt+\frac{1-t^{-2m}%
}{1-t^{-2}})
\end{equation}

Then we immediately obtain the following theorem, which verify the expansion
formula of Poincare polynomial of Heegaard-Floer homology

\begin{theorem}
The expansion formula (\ref{Heegaard-Floer Expansion}) is valid for untwisted Whitehead double $%
D(T(2,2m+1))$ of torus knot $T(2,2m+1)$. In fact, the invariant $\gamma
(D(T(2,2m+1)))=1$ and the Poincare polynomial have the expansion predicted
in the conjecture with coefficient $KF(D(T(2,2m+1));q,t)=-2mt-2\frac{%
1-t^{-2m}}{1-t^{-2}}$.
\end{theorem}

\section{Appendix}

\subsection{Cyclotomic expansions for triply-graded reduced HOMFLY-PT
homology}

For instance, we have the following expansion for $N=1$ and $2$.%
\begin{equation}
(-t)^{\alpha (\mathcal{K})}\mathcal{P}_{1}(\mathcal{K};a,q,t)=1+H_{1}(%
\mathcal{K};a,q,t)(aq^{-1}+t^{-1}a^{-1}q^{1})(t^{2}aq+t^{-1}a^{-1}q^{-1})
\end{equation}

and

\begin{eqnarray}
(-t)^{2\alpha (\mathcal{K})}\mathcal{P}_{2}(\mathcal{K};a,q,t) &=&1+H_{1}(%
\mathcal{K}%
;a,q,t)(aq^{-1}+t^{-1}a^{-1}q^{1})(q+q^{-1})(t^{2}aq^{2}+t^{-1}a^{-1}q^{-2})
\notag \\
&&+H_{2}(\mathcal{K}%
;a,q,t)(aq^{-1}+t^{-1}a^{-1}q^{1})(t^{2}aq^{2}+t^{-1}a^{-1}q^{-2})(t^{2}aq^{3}+t^{-1}a^{-1}q^{-3})
\end{eqnarray}

\bigskip

We test the expression of knots $3_{1}-7_{7}$ obtained in \cite{DGR}, which
are quasi-alternating knots. Here we just explicitly provide their value
for $H_{1}(\mathcal{K},a,q,t)$.

\bigskip

$%
\begin{array}{llllll}
\mathcal{K} & \sigma  & g_{4} & s & \alpha  & H_{1}(\mathcal{K},a,q,t) \\
3_{1} & -2 & 1 & 2 & 1 & -a^{2}t^{2} \\
&  &  &  &  &  \\
4_{1} & 0 & 0 & 0 & 0 & 1 \\
&  &  &  &  &  \\
5_{1} & -4 & 2 & 4 & 2 & -a^{2}t^{2}+a^{4}q^{\underline{2}%
}t^{3}+a^{4}q^{2}t^{5} \\
&  &  &  &  &  \\
5_{2} & -2 & 1 & 2 & 1 & -a^{2}t^{2}-a^{4}t^{4} \\
&  &  &  &  &  \\
6_{1} & 0 & 0 & 0 & 0 & 1+a^{2}t^{2} \\
&  &  &  &  &  \\
6_{2} & -2 & 1 & 2 & 1 & -a^{2}q^{\underline{2}%
}t^{1}-a^{2}t^{2}-a^{2}q^{2}t^{3} \\
&  &  &  &  &  \\
6_{3} & 0 & 0 & 0 & 0 & q^{\underline{2}}t^{\underline{1}}+1+q^{2}t^{1} \\
&  &  &  &  &  \\
7_{1} & -6 & 3 & 6 & 3 & -a^{2}t^{2}+a^{4}q^{\underline{2}%
}t^{3}+a^{4}q^{2}t^{5}-a^{6}q^{\underline{4}}t^{4}-a^{6}t^{6}-a^{6}q^{4}t^{8}
\\
&  &  &  &  &  \\
7_{2} & -2 & 1 & 2 & 1 & -a^{2}t^{2}-a^{4}t^{4}-a^{6}t^{6} \\
&  &  &  &  &  \\
7_{3} & 4 & -2 & -4 & -2 & a^{\underline{6}}q^{\underline{2}}t^{\underline{7}%
}+a^{\underline{6}}q^{2}t^{\underline{5}}+a^{\underline{4}}q^{\underline{2}%
}t^{\underline{5}}+a^{\underline{4}}q^{2}t^{\underline{3}}-a^{\underline{2}%
}t^{\underline{2}} \\
&  &  &  &  &  \\
7_{4} & 2 & -1 & -2 & -1 & -a^{\underline{6}}t^{\underline{6}}-2a^{%
\underline{4}}t^{\underline{4}}-a^{\underline{2}}t^{\underline{2}} \\
&  &  &  &  &  \\
7_{5} & -4 & 2 & 4 & 2 & -a^{2}t^{2}+a^{4}q^{\underline{2}%
}t^{3}+a^{4}q^{2}t^{5}+a^{6}q^{\underline{2}}t^{5}+a^{6}t^{6}+a^{6}q^{2}t^{7}
\\
&  &  &  &  &  \\
7_{6} & -2 & 1 & 2 & 1 & -a^{2}q^{\underline{2}%
}t^{1}-2a^{2}t^{2}-a^{2}q^{2}t^{3}-a^{4}t^{4} \\
&  &  &  &  &  \\
7_{7} & 0 & 0 & 0 & 0 & a^{\underline{2}}t^{\underline{2}}+q^{\underline{2}%
}t^{\underline{1}}+2+q^{2}t^{1}%
\end{array}%
$

\begin{remark}
Notations $\sigma $, $g_{4}$ and $s$ stands for the signature, smooth 4-ball
genus and Rasmussen $s$ invaraint of a knot $\mathcal{K}$ respectively, and $%
a^{\underline{u}}q^{\underline{v}}t^{\underline{w}}$ denotes term $%
a^{-u}q^{-v}t^{-w}$.
\end{remark}

\bigskip

We test more homologically thick knots. Expression of $8_{19}$ and $9_{42}$
obtained in \cite{GS}(We make a variable change $a\rightarrow a^{-2}$, $%
q\rightarrow q^{-2}$, and $t\rightarrow t^{-1}$, because they use mirror
knot). Knots $10_{124}$, $10_{128}$, $10_{132}$, $10_{136}$, $10_{139}$, $%
10_{145}$, $10_{152}$, $10_{153}$, $10_{154}$ and $10_{161}$ are obtained in
pp.42-45 \cite{DGR}.

We converted dotted diagrams shown in pp.42-45 \cite{DGR} to the following
table.

$%
\begin{array}{ll}
\mathcal{K} & \mathcal{P}_{1}(\mathcal{K},a,q,t) \\
8_{19} &
\begin{array}{l}
a^{\underline{10}}t^{\underline{8}}+a^{\underline{8}}(q^{4}t^{\underline{3}%
}+t^{\underline{5}}+q^{\underline{4}}t^{\underline{7}}+q^{2}t^{\underline{5}%
}+q^{\underline{2}}t^{\underline{7}}) \\
+a^{\underline{6}}(q^{6}+q^{2}t^{\underline{2}}+q^{\underline{2}}t^{%
\underline{4}}+q^{\underline{6}}t^{\underline{6}}+t^{\underline{4}})%
\end{array}
\\
&  \\
9_{42} & a^{\underline{2}}(q^{\underline{2}}t^{\underline{4}}+q^{2}t^{%
\underline{2}})+(q^{\underline{4}}t^{\underline{3}}+2t^{\underline{1}%
}+1+q^{4}t^{1})+a^{2}(q^{\underline{2}}+q^{2}t^{2}) \\
&  \\
10_{124} &
\begin{array}{l}
a^{\underline{12}}(q^{\underline{2}}t^{\underline{10}}+q^{2}t^{\underline{8}%
})+a^{\underline{10}}(q^{\underline{6}}t^{\underline{9}}+q^{\underline{4}}t^{%
\underline{9}}+q^{\underline{2}}t^{\underline{7}}+2t^{\underline{7}}+q^{2}t^{%
\underline{5}}+q^{4}t^{\underline{5}} \\
+q^{6}t^{\underline{3}})+a^{\underline{8}}(q^{\underline{8}}t^{\underline{8}%
}+q^{\underline{4}}t^{\underline{6}}+q^{\underline{2}}t^{\underline{6}}+t^{%
\underline{4}}+q^{2}t^{\underline{4}}+q^{4}t^{\underline{2}}+q^{8})%
\end{array}
\\
&  \\
10_{128} &
\begin{array}{l}
a^{\underline{12}}t^{\underline{10}}+a^{\underline{10}}(q^{\underline{4}}t^{%
\underline{9}}+q^{\underline{2}}t^{\underline{9}}+t^{\underline{7}}+t^{%
\underline{8}}+q^{2}t^{\underline{7}}+q^{4}t^{\underline{5}})+a^{\underline{8%
}}(q^{\underline{6}}t^{\underline{8}} \\
+q^{\underline{4}}t^{\underline{7}}+2q^{\underline{2}}t^{\underline{6}}+q^{%
\underline{2}}t^{\underline{7}}+t^{\underline{5}}+t^{\underline{6}}+2q^{2}t^{%
\underline{4}}+q^{2}t^{\underline{5}}+q^{4}t^{\underline{3}}+q^{6}t^{%
\underline{2}}) \\
+a^{\underline{6}}(q^{\underline{6}}t^{\underline{6}}+q^{\underline{4}}t^{%
\underline{5}}+q^{\underline{2}}t^{\underline{4}}+t^{\underline{3}}+t^{%
\underline{4}}+q^{2}t^{\underline{2}}+q^{4}t^{\underline{1}}+q^{6})%
\end{array}
\\
&  \\
10_{132} & a^{2}(q^{\underline{2}}+q^{\underline{2}%
}t^{1}+q^{2}t^{2}+q^{2}t^{3})+a^{4}(q^{\underline{4}%
}t^{2}+t^{3}+2t^{4}+q^{4}t^{6})+a^{6}(q^{\underline{2}}t^{5}+q^{2}t^{7}) \\
&  \\
10_{136} &
\begin{array}{l}
a^{\underline{4}}t^{\underline{3}}+a^{\underline{2}}(2q^{\underline{2}}t^{%
\underline{2}}+t^{\underline{1}}+2q^{2})+(q^{\underline{4}}t^{\underline{1}%
}+q^{\underline{2}} \\
+3t^{1}+1+q^{2}t^{2}+q^{4}t^{3})+a^{2}(q^{\underline{2}%
}t^{2}+t^{3}+q^{2}t^{4})%
\end{array}
\\
&  \\
10_{139} &
\begin{array}{l}
a^{\underline{12}}(q^{\underline{2}}t^{\underline{10}}+t^{\underline{9}%
}+2q^{2}t^{\underline{8}})+a^{\underline{10}}(q^{\underline{6}}t^{\underline{%
9}}+q^{\underline{4}}t^{\underline{9}}+q^{\underline{2}}t^{\underline{8}}+q^{%
\underline{2}}t^{\underline{7}}+2t^{\underline{7}}+q^{2}t^{\underline{6}} \\
+q^{2}t^{\underline{5}}+q^{4}t^{\underline{5}}+q^{6}t^{\underline{3}})+a^{%
\underline{8}}(q^{\underline{8}}t^{\underline{8}}+q^{\underline{4}}t^{%
\underline{6}}+q^{\underline{2}}t^{\underline{6}}+t^{\underline{5}}+t^{%
\underline{4}}+q^{2}t^{\underline{4}}+q^{4}t^{\underline{2}}+q^{8})%
\end{array}
\\
&  \\
10_{145} &
\begin{array}{l}
a^{4}(q^{\underline{4}}+t^{2}+t^{3}+q^{4}t^{4})+a^{6}(q^{\underline{2}%
}t^{3}+q^{\underline{2}}t^{4}+ \\
t^{5}+q^{2}t^{5}+q^{2}t^{6})+a^{8}(q^{\underline{2}%
}t^{6}+t^{7}+q^{2}t^{8})+a^{10}t^{9}%
\end{array}
\\
&  \\
10_{152} &
\begin{array}{l}
a^{8}(q^{\underline{8}}+q^{\underline{4}}t^{2}+2q^{\underline{2}%
}t^{4}+t^{4}+t^{5}+2q^{2}t^{6}+q^{4}t^{6}+q^{8}t^{8})+a^{10}(q^{\underline{6}%
}t^{3}+2q^{\underline{4}}t^{5} \\
+q^{\underline{2}}t^{5}+q^{\underline{2}%
}t^{6}+4t^{7}+q^{2}t^{7}+q^{2}t^{8}+2q^{4}t^{9}+q^{6}t^{9})+a^{12}(2q^{%
\underline{2}}t^{8}+t^{9}+2q^{2}t^{10})%
\end{array}
\\
&  \\
10_{153} &
\begin{array}{l}
a^{\underline{2}}(q^{\underline{4}}t^{\underline{5}}+t^{\underline{3}%
}+q^{4}t^{\underline{1}})+(q^{\underline{6}}t^{\underline{4}}+2q^{\underline{%
2}}t^{\underline{2}}+q^{\underline{2}}t^{\underline{1}%
}+2+2q^{2}+q^{2}t^{1}+q^{6}t^{2}) \\
+a^{2}(q^{\underline{4}}t^{\underline{1}}+q^{\underline{4}}+q^{\underline{2}%
}t^{1}+t^{1}+2t^{2}+q^{2}t^{3}+q^{4}t^{3}+q^{4}t^{4})+a^{4}(q^{\underline{2}%
}t^{3}+t^{4}+q^{2}t^{5})%
\end{array}
\\
&  \\
10_{154} &
\begin{array}{l}
a^{\underline{12}}t^{\underline{10}}+a^{\underline{10}}(2q^{\underline{2}}t^{%
\underline{9}}+2t^{\underline{8}}+2q^{2}t^{\underline{7}})+a^{\underline{8}%
}(q^{\underline{4}}t^{\underline{7}}+q^{\underline{4}}t^{\underline{8}}+2q^{%
\underline{2}}t^{\underline{7}}+3t^{\underline{6}}+t^{\underline{5}%
}+2q^{2}t^{\underline{5}} \\
+q^{4}t^{\underline{3}}+q^{4}t^{\underline{4}})+a^{\underline{6}}(q^{%
\underline{6}}t^{\underline{6}}+q^{\underline{2}}t^{\underline{4}}+q^{%
\underline{2}}t^{\underline{5}}+2t^{\underline{4}}+q^{2}t^{\underline{2}%
}+q^{2}t^{\underline{3}}+q^{6})%
\end{array}
\\
&  \\
10_{161} &
\begin{array}{l}
a^{6}(q^{\underline{6}}+q^{\underline{2}}t^{2}+q^{\underline{2}%
}t^{3}+t^{4}+q^{2}t^{4}+q^{2}t^{5}+q^{6}t^{6})+a^{8}(q^{\underline{4}%
}t^{3}+q^{\underline{4}}t^{4} \\
+q^{\underline{2}%
}t^{5}+t^{5}+2t^{6}+q^{2}t^{7}+q^{4}t^{7}+q^{4}t^{8})+a^{10}(q^{\underline{2}%
}t^{7}+t^{8}+q^{2}t^{9})%
\end{array}%
\end{array}%
$

\bigskip

We list the following table for the coefficient $H_{1}(\mathcal{K},a,q,t)$ in
the expansion.

\bigskip $\bigskip
\begin{array}{llllll}
\mathcal{K} & \sigma  & g_{4} & s & \alpha  & H_{1}(\mathcal{K},a,q,t) \\
8_{19} & 6 & 3 & 6 & -3 & -a^{\underline{8}}t^{\underline{9}}-a^{\underline{6%
}}q^{\underline{4}}t^{\underline{8}}-a^{\underline{6}}t^{\underline{6}}+a^{%
\underline{4}}q^{\underline{2}}t^{\underline{5}}-a^{\underline{6}}q^{4}t^{%
\underline{4}}+a^{\underline{4}}q^{2}t^{\underline{3}}-a^{\underline{2}}t^{%
\underline{2}} \\
&  &  &  &  &  \\
9_{42} & 2 & 1 & 0 & 0 & q^{\underline{2}}t^{\underline{2}}+q^{2} \\
&  &  &  &  &  \\
10_{124} & 8 & 4 & 8 & -4 &
\begin{array}{l}
a^{\underline{10}}q^{\underline{2}}t^{\underline{12}}+a^{\underline{10}%
}q^{2}t^{\underline{10}}+a^{\underline{8}}q^{\underline{6}}t^{\underline{11}%
}+a^{\underline{8}}q^{\underline{2}}t^{\underline{9}}+a^{\underline{8}%
}q^{2}t^{\underline{7}}+a^{\underline{8}}q^{6}t^{\underline{5}} \\
-a^{\underline{6}}q^{\underline{4}}t^{\underline{8}}-a^{\underline{6}}t^{%
\underline{6}}-a^{\underline{6}}q^{4}t^{\underline{4}}+a^{\underline{4}}q^{%
\underline{2}}t^{\underline{5}}+a^{\underline{4}}q^{2}t^{\underline{3}}-a^{%
\underline{2}}t^{\underline{2}}%
\end{array}
\\
&  &  &  &  &  \\
10_{128} & 6 & 3 & 6 & -3 &
\begin{array}{l}
-a^{\underline{10}}t^{\underline{11}}-a^{\underline{8}}q^{\underline{4}}t^{%
\underline{10}}-a^{\underline{8}}t^{\underline{9}}-a^{\underline{8}}t^{%
\underline{8}}-a^{\underline{8}}q^{4}t^{\underline{6}} \\
-a^{\underline{6}}q^{\underline{4}}t^{\underline{8}}-a^{\underline{6}}t^{%
\underline{6}}-a^{\underline{6}}q^{4}t^{\underline{4}}+a^{\underline{4}}q^{%
\underline{2}}t^{\underline{5}}+a^{\underline{4}}q^{2}t^{\underline{3}}-a^{%
\underline{2}}t^{\underline{2}}%
\end{array}
\\
&  &  &  &  &  \\
10_{132} & 0 & 1 & -2 & 1 & -a^{2}t^{2}-a^{4}q^{\underline{2}%
}t^{4}-a^{4}q^{2}t^{6} \\
&  &  &  &  &  \\
10_{136} & 2 & 1 & 0 & 0 & a^{\underline{2}}t^{\underline{1}}+q^{\underline{2%
}}+t^{1}+q^{2}t^{2} \\
&  &  &  &  &  \\
10_{139} & 6 & 4 & 8 & -4 &
\begin{array}{l}
a^{\underline{10}}q^{\underline{2}}t^{\underline{12}}+a^{\underline{10}}t^{%
\underline{11}}+a^{\underline{10}}q^{2}t^{\underline{10}}+a^{\underline{8}%
}q^{\underline{6}}t^{\underline{11}}+a^{\underline{8}}q^{\underline{2}}t^{%
\underline{9}}+a^{\underline{8}}q^{2}t^{\underline{7}} \\
+a^{\underline{8}}q^{6}t^{\underline{5}}-a^{\underline{6}}q^{\underline{4}%
}t^{\underline{8}}-a^{\underline{6}}t^{\underline{6}}-a^{\underline{6}%
}q^{4}t^{\underline{4}}+a^{\underline{4}}q^{\underline{2}}t^{\underline{5}%
}+a^{\underline{4}}q^{2}t^{\underline{3}}-a^{\underline{2}}t^{\underline{2}}%
\end{array}
\\
&  &  &  &  &  \\
10_{145} & -2 & 2 & -4 & 2 & -a^{2}t^{2}+a^{4}q^{\underline{2}%
}t^{3}+a^{4}q^{2}t^{5}+a^{6}t^{7}+a^{8}t^{9} \\
&  &  &  &  &  \\
10_{152} & -6 & 4 & -8 & 4 &
\begin{array}{l}
-a^{2}t^{2}+a^{4}q^{\underline{2}}t^{3}+a^{4}q^{2}t^{5}-a^{6}q^{\underline{4}%
}t^{4}-a^{6}t^{6}-a^{6}q^{4}t^{8}+a^{8}q^{\underline{6}}t^{5} \\
+a^{8}q^{\underline{2}}t^{7}+a^{8}q^{2}t^{9}+a^{8}q^{6}t^{11}+2a^{10}q^{%
\underline{2}}t^{10}+a^{10}t^{11}+2a^{10}q^{2}t^{12}%
\end{array}
\\
&  &  &  &  &  \\
10_{153} & 0 & 0 & 0 & 0 & q^{\underline{4}}t^{\underline{3}}+t^{\underline{1%
}}+q^{4}t^{1}+a^{2}q^{\underline{2}}t^{1}+a^{2}t^{2}+a^{2}q^{2}t^{3} \\
&  &  &  &  &  \\
10_{154} & 4 & 3 & 6 & -3 &
\begin{array}{l}
-a^{\underline{10}}t^{\underline{11}}-a^{\underline{8}}q^{\underline{2}}t^{%
\underline{10}}-2a^{\underline{8}}t^{\underline{9}}-a^{\underline{8}}q^{2}t^{%
\underline{8}}-a^{\underline{6}}q^{\underline{4}}t^{\underline{8}} \\
-a^{\underline{6}}t^{\underline{6}}-a^{\underline{6}}q^{4}t^{\underline{4}%
}+a^{\underline{4}}q^{\underline{2}}t^{\underline{5}}+a^{\underline{4}%
}q^{2}t^{\underline{3}}-a^{\underline{2}}t^{\underline{2}}%
\end{array}
\\
&  &  &  &  &  \\
10_{161} & -4 & 3 & -6 & 3 &
\begin{array}{l}
-a^{2}t^{2}+a^{4}q^{\underline{2}}t^{3}+a^{4}q^{2}t^{5}-a^{6}q^{\underline{4}%
}t^{4} \\
-a^{6}t^{6}-a^{6}q^{4}t^{8}-a^{8}q^{\underline{2}%
}t^{8}-a^{8}t^{9}-a^{8}q^{2}t^{10}%
\end{array}%
\end{array}%
$

\begin{remark}
For these values, $\alpha (\mathcal{K})$ is coincide with the Ozsv\'{a}%
th-Szab\'{o}'s $\tau $ invariant and Rasmussen's $s$ invariant up to a
factor of $2$.
\end{remark}

\bigskip

We also tested higher representation for knots $3_{1}$, $5_{1}$ and $7_{1}$
obtained in (3.61) of \cite{FGS1}(We make a variable change $q\rightarrow
q^{2}$, and $t\rightarrow t^{2}$), knots $4_{1}$ obtained in (2.12) of \cite%
{FGS2}(original in \cite{IMMM}), $5_{2}$ and $6_{1}$ in \cite{GS} and knots $%
8_{19}$ and $9_{42}$ obtained in Appendix B of \cite{GS} (We make a variable
change $q\rightarrow q^{-2}$, and $t\rightarrow t^{-2}$, because they use
mirror knot)

$%
\begin{array}{llll}
\mathcal{K} & \sigma (\mathcal{K}) & \alpha (\mathcal{K}) & H_{2}(\mathcal{K}%
,a,q,t) \\
3_{1} & -2 & 1 & (a+t^{-1}a^{-1})a^{4}q^{2}t^{4} \\
&  &  &  \\
4_{1} & 0 & 0 & (a+t^{-1}a^{-1}) \\
&  &  &  \\
5_{1} & -4 & 2 &
\begin{array}{l}
a^{3}q^{2}t^{3}-a^{5}t^{4}-a^{5}q^{4}t^{6}-a^{5}q^{6}t^{6}-a^{7}t^{5}+a^{7}q^{10}t^{9}
\\
+a^{9}q^{\underline{2}%
}t^{6}+a^{9}q^{4}t^{8}+a^{9}q^{6}t^{8}+a^{9}q^{10}t^{10}%
\end{array}
\\
&  &  &  \\
5_{2} & -2 & 1 &
(a+t^{-1}a^{-1})(a^{4}q^{2}t^{4}+a^{6}q^{2}t^{6}+a^{6}q^{4}t^{6}+a^{8}q^{6}t^{8})
\\
&  &  &  \\
6_{1} & 0 & 0 &
(a+t^{-1}a^{-1})(1+a^{2}t^{2}+a^{2}q^{2}t^{2}+a^{4}q^{4}t^{4}) \\
&  &  &  \\
6_{2} & -2 & 1 &
\begin{array}{l}
a^{3}q^{\underline{4}%
}t^{1}+a^{3}t^{2}+a^{3}q^{2}t^{2}+2a^{3}q^{2}t^{3}+a^{3}q^{4}t^{3}+a^{3}q^{4}t^{4}+a^{3}q^{6}t^{4}+a^{3}q^{8}t^{5}
\\
+a^{5}q^{\underline{4}}t^{2}+a^{5}t^{3}+a^{5}q^{\underline{2}%
}t^{3}+2a^{5}q^{2}t^{4}+a^{5}q^{4}t^{4}+a^{5}q^{4}t^{5}+a^{5}q^{6}t^{5}+a^{5}q^{8}t^{6}%
\end{array}
\\
&  &  &  \\
6_{3} & 0 & 0 &
\begin{array}{l}
a^{\underline{1}}q^{\underline{6}}t^{\underline{3}}+a^{\underline{1}}q^{%
\underline{4}}t^{\underline{2}}+a^{\underline{1}}q^{\underline{2}}t^{%
\underline{2}}+2a^{\underline{1}}t^{\underline{1}}+a^{\underline{1}}q^{2}t^{%
\underline{1}}+a^{\underline{1}}q^{2}+a^{\underline{1}}q^{4}+a^{\underline{1}%
}q^{6}t^{1} \\
+a^{1}q^{\underline{6}}t^{\underline{2}}+a^{1}q^{\underline{4}}t^{\underline{%
1}}+a^{1}q^{\underline{2}}t^{\underline{1}}+2a^{1}+a^{1}q^{\underline{2}%
}+a^{1}q^{2}t^{1}+a^{1}q^{4}t^{1}+a^{1}q^{6}t^{2}%
\end{array}
\\
&  &  &  \\
7_{1} & -6 & 3 &
\begin{array}{l}
a^{3}q^{2}t^{3}-a^{5}t^{4}-a^{5}q^{4}t^{6}-a^{5}q^{6}t^{6}+a^{7}q^{%
\underline{2}}t^{\underline{5}%
}+a^{7}q^{2}t^{7}+a^{7}q^{4}t^{7}+a^{7}q^{6}t^{9} \\
+a^{7}q^{8}t^{9}+a^{7}q^{10}t^{9}+a^{9}q^{\underline{2}%
}t^{6}-a^{9}q^{8}t^{10}-a^{9}q^{12}t^{12}-a^{9}q^{14}t^{12}-a^{11}q^{%
\underline{4}}t^{7} \\
-a^{11}q^{2}t^{9}-a^{11}q^{4}t^{9}-a^{11}q^{8}t^{11}+a^{11}q^{18}t^{15}+a^{13}q^{%
\underline{6}}t^{8}+a^{13}t^{10}+a^{13}q^{2}t^{10} \\
+a^{13}q^{6}t^{12}+a^{13}q^{8}t^{12}+a^{13}q^{10}t^{12}+a^{13}q^{12}t^{14}+a^{13}q^{14}t^{14}+a^{13}q^{18}t^{16}%
\end{array}
\\
&  &  &  \\
8_{19} & 6 & -3 &
\begin{array}{l}
a^{\underline{17}}q^{\underline{12}}t^{\underline{19}}+a^{\underline{15}}q^{%
\underline{16}}t^{\underline{18}}+a^{\underline{15}}q^{\underline{14}}t^{%
\underline{18}}+a^{\underline{15}}q^{\underline{12}}t^{\underline{18}}+a^{%
\underline{15}}q^{\underline{10}}t^{\underline{16}}+a^{\underline{15}}q^{%
\underline{8}}t^{\underline{16}}+a^{\underline{15}}q^{\underline{4}}t^{1%
\underline{4}} \\
+a^{\underline{15}}q^{\underline{2}}t^{1\underline{4}}+a^{\underline{13}}q^{%
\underline{18}}t^{\underline{17}}+a^{\underline{13}}q^{\underline{16}}t^{%
\underline{17}}+a^{\underline{13}}q^{1\underline{4}}t^{\underline{17}}+a^{%
\underline{13}}q^{\underline{14}}t^{\underline{15}}+a^{\underline{13}}q^{%
\underline{12}}t^{\underline{15}} \\
+a^{\underline{13}}q^{\underline{10}}t^{\underline{13}}+a^{\underline{13}}q^{%
\underline{8}}t^{\underline{13}}+a^{\underline{13}}q^{\underline{6}}t^{%
\underline{13}}+a^{\underline{13}}q^{\underline{2}}t^{\underline{11}}+a^{%
\underline{13}}t^{\underline{11}}+a^{\underline{13}}q^{6}t^{\underline{9}%
}+a^{\underline{11}}q^{\underline{18}}t^{\underline{16}} \\
-a^{\underline{11}}q^{\underline{10}}t^{\underline{14}}-a^{\underline{11}}q^{%
\underline{8}}t^{1\underline{4}}-a^{\underline{11}}q^{\underline{8}}t^{%
\underline{12}}-a^{\underline{11}}q^{\underline{4}}t^{\underline{10}}-a^{%
\underline{11}}q^{\underline{2}}t^{\underline{10}}-a^{\underline{11}}q^{4}t^{%
\underline{8}}-a^{\underline{9}}q^{\underline{14}}t^{\underline{13}} \\
-a^{\underline{9}}q^{\underline{12}}t^{\underline{13}}-a^{\underline{9}}q^{%
\underline{8}}t^{\underline{11}}+a^{\underline{9}}q^{\underline{4}}t^{%
\underline{11}}+a^{\underline{9}}q^{\underline{2}}t^{\underline{11}}+a^{%
\underline{9}}q^{2}t^{\underline{7}}+a^{\underline{7}}q^{\underline{10}}t^{%
\underline{10}}+a^{\underline{7}}q^{\underline{8}}t^{\underline{10}} \\
+a^{\underline{7}}q^{\underline{6}}t^{\underline{10}}+a^{\underline{7}}q^{%
\underline{4}}t^{\underline{8}}+a^{\underline{7}}q^{\underline{2}}t^{%
\underline{8}}+a^{\underline{7}}q^{2}t^{\underline{6}}-a^{\underline{5}}q^{%
\underline{6}}t^{\underline{7}}-a^{\underline{5}}q^{\underline{4}}t^{%
\underline{7}}-a^{\underline{5}}t^{\underline{5}}+a^{\underline{3}}q^{%
\underline{2}}t^{\underline{4}}%
\end{array}
\\
&  &  &  \\
9_{42} & 2 & 0 &
\begin{array}{l}
a^{\underline{1}}q^{\underline{8}}t^{\underline{5}}+a^{\underline{1}}q^{%
\underline{6}}t^{\underline{5}}+a^{\underline{1}}q^{\underline{6}}t^{%
\underline{4}}+a^{\underline{1}}q^{\underline{4}}t^{\underline{5}}+a^{%
\underline{1}}q^{\underline{4}}t^{\underline{4}}+2a^{1}q^{\underline{4}}t^{%
\underline{3}}+a^{\underline{1}}q^{\underline{2}}t^{\underline{3}} \\
+2a^{\underline{1}}t^{\underline{3}}+a^{\underline{1}}t^{\underline{2}}+a^{%
\underline{1}}q^{2}t^{\underline{3}}+a^{\underline{1}}q^{2}t^{\underline{2}%
}+a^{\underline{1}}q^{4}t^{\underline{1}}+a^{1}q^{\underline{8}}t^{%
\underline{4}}+a^{1}q^{\underline{6}}t^{\underline{4}}+a^{1}q^{\underline{6}%
}t^{\underline{3}} \\
+a^{1}q^{\underline{4}}t^{\underline{4}}+a^{1}q^{\underline{2}}t^{\underline{%
3}}+a^{1}q^{\underline{4}}t^{\underline{2}}+2a^{1}q^{\underline{2}}t^{%
\underline{2}}+a^{1}q^{\underline{2}}t^{\underline{1}}+3a^{1}t^{\underline{2}%
}+a^{1}t^{\underline{1}} \\
+2a^{1}q^{2}t^{\underline{2}}+2a^{1}q^{2}t^{\underline{1}%
}+a^{1}q^{2}+2a^{1}q^{4}t^{\underline{1}}+2a^{1}q^{4}+a^{3}q^{\underline{4}%
}t^{\underline{2}}+a^{3}q^{\underline{4}}t^{\underline{1}} \\
+a^{3}q^{\underline{2}}t^{\underline{2}}+a^{3}q^{\underline{2}}t^{\underline{%
1}}+a^{3}t^{\underline{1}}+a^{3}q^{2}t^{\underline{1}%
}+2a^{3}q^{2}+a^{3}q^{2}t^{1}+2a^{3}q^{4}+a^{3}q^{4}t^{1}%
\end{array}%
\end{array}%
$

\begin{remark}
Careful reader may find only $H_{2}(\mathcal{K},a,q,t)$ of knot $3_{1}$, $%
4_{1}$, $5_{2}$, $6_{1}$ has an additional factor $(a+t^{-1}a^{-1})$, while
other don't have. That's the reason why we can not make the conjucture one
step further.
\end{remark}

\subsection{Volume Conjecture for SU(n) specialized superpolynomial of
colored HOMFLY-PT homology}

Here we provide several tables of $2\pi \log \left( \frac{\mathcal{P}_{N}(%
\mathcal{K};q^{n},q,q^{-(N+n-1)})|_{q=e^{\frac{\pi \sqrt{-1}}{N+b}}}}{%
\mathcal{P}_{N-1}(\mathcal{K};q^{n},q,q^{-(N+n-2)})|_{q=e^{\frac{\pi \sqrt{-1%
}}{N-1+b}}}}\right) $ for knot $\mathcal{K=}5_{2}$.

\bigskip

$%
\begin{array}{cccc}
N\backslash (n,b) & (2,1) & (2,2) & (2,3) \\
10 & 3.73795+2.62595\sqrt{-1} & 4.72339+2.30778\sqrt{-1} & 5.57612+2.02747%
\sqrt{-1} \\
20 & 3.27786+2.92530\sqrt{-1} & 3.84449+2.81820\sqrt{-1} & 4.36265+2.71525%
\sqrt{-1} \\
30 & 3.13249+2.97960\sqrt{-1} & 3.52355+2.92820\sqrt{-1} & 3.89157+2.87605%
\sqrt{-1} \\
40 & 3.05822+2.99885\sqrt{-1} & 3.35658+2.96886\sqrt{-1} & 3.64157+2.93753%
\sqrt{-1} \\
50 & 3.01308+3.00786\sqrt{-1} & 3.25424+2.98824\sqrt{-1} & 3.48668+2.96737%
\sqrt{-1} \\
70 & 2.96096+3.01577\sqrt{-1} & 3.13525+3.00551\sqrt{-1} & 3.30500+2.99436%
\sqrt{-1} \\
100 & 2.92148+3.02001\sqrt{-1} & 3.04458+3.01489\sqrt{-1} & 3.16541+3.00924%
\sqrt{-1} \\
150 & 2.89056+3.02229\sqrt{-1} & 2.97319+3.01998\sqrt{-1} & 3.05480+3.01740%
\sqrt{-1}%
\end{array}%
$

\bigskip

$%
\begin{array}{cccc}
N\backslash (n,b) & (3,1) & (3,2) & (3,3) \\
10 & 3.78463+2.47268\sqrt{-1} & 4.77077+2.02852\sqrt{-1} & 5.60791+1.65764%
\sqrt{-1} \\
20 & 3.29553+2.89137\sqrt{-1} & 3.85936+2.74299\sqrt{-1} & 4.37499+2.60602%
\sqrt{-1} \\
30 & 3.14046+2.96457\sqrt{-1} & 3.53064+2.89368\sqrt{-1} & 3.89784+2.82443%
\sqrt{-1} \\
40 & 3.06273+2.99039\sqrt{-1} & 3.36071+2.94911\sqrt{-1} & 3.64534+2.90755%
\sqrt{-1} \\
50 & 3.01598+3.00244\sqrt{-1} & 3.25694+2.97547\sqrt{-1} & 3.48919+2.94781%
\sqrt{-1} \\
70 & 2.96244+3.01301\sqrt{-1} & 3.13666+2.99891\sqrt{-1} & 3.30634+2.98415%
\sqrt{-1} \\
100 & 2.92221+3.01866\sqrt{-1} & 3.04528+3.01163\sqrt{-1} & 3.16609+3.00414%
\sqrt{-1} \\
150 & 2.89089+3.02169\sqrt{-1} & 2.97351+3.01852\sqrt{-1} & 3.05511+3.01511%
\sqrt{-1}%
\end{array}%
$

\bigskip

$%
\begin{array}{ccccc}
N\backslash (n,b) & (4,1) & (4,2) & (4,3) & (4,4) \\
10 & \# & 4.90771+1.80487\sqrt{-1} & 5.70074+1.31854\sqrt{-1} &
6.38359+0.946278\sqrt{-1} \\
20 & \# & 3.90345+2.68899\sqrt{-1} & 4.41159+2.51307\sqrt{-1} &
4.87781+2.35457\sqrt{-1} \\
30 & \# & 3.55180+2.86986\sqrt{-1} & 3.91657+2.78186\sqrt{-1} &
4.26043+2.69779\sqrt{-1} \\
40 & \# & 3.37306+2.93576\sqrt{-1} & 3.65662+2.88323\sqrt{-1} &
3.92773+2.83155\sqrt{-1} \\
50 & \# & 3.26502+2.96695\sqrt{-1} & 3.49671+2.93210\sqrt{-1} &
3.72016+2.89720\sqrt{-1} \\
70 & \# & 3.14089+2.99458\sqrt{-1} & 3.31036+2.97605\sqrt{-1} &
3.47547+2.95712\sqrt{-1} \\
100 & \# & 3.04739+3.00951\sqrt{-1} & 3.16812+3.00014\sqrt{-1} &
3.28665+2.99042\sqrt{-1} \\
150 & \# & 2.97446+3.01758\sqrt{-1} & 3.05604+3.01332\sqrt{-1} &
3.13662+3.00884\sqrt{-1}%
\end{array}%
$

\subsection{Expansion formula for Poincare polynomial of Heegaard-Floer
homology}

For a knot $\mathcal{K}$, there exists an integer valued invariant $\gamma (%
\mathcal{K})\in
\mathbb{Z}
$ of a knot $\mathcal{K}$, s.t. Poincare polynomial $HFK(\mathcal{K};q^{2},t)
$ of Heegaard-Floer knot homology of a knot $\mathcal{K}$ has the following
expansion formula%
\begin{equation}
(-t)^{\gamma (\mathcal{K})}HFK(\mathcal{K};q^{2},t)=1+KF(\mathcal{K}%
;q,t)(q+t^{-1}q^{-1})^{2}
\end{equation}

with coefficient functions $KF(\mathcal{K};q,t)\in
\mathbb{Z}
\lbrack q^{\pm 1},t^{\pm 1}]$.

\bigskip

We test the expression of homologically thick knots $8_{19}$, $9_{42}$, $%
10_{124}$, $10_{128}$, $10_{132}$, $10_{136}$, $10_{139}$, $10_{145}$, $%
10_{152}$, $10_{153}$, $10_{154}$, $10_{161}$ obtained in \cite{BG}.

(From knot $10_{124}$, we make a variable change $t\rightarrow q^{-2}$, and $%
q\rightarrow t^{-1}$. For knot $10_{124}$, we use $%
q^{8}t^{8}+q^{6}t^{7}+q^{2}t^{4}+t^{3}+q^{-2}t^{2}+q^{-6}t+q^{-8}$ instead
of $q^{-8}t^{-4}+q^{-7}t^{-3}+q^{-4}t^{-1}+q^{-3}+q^{-2}t+q^{-1}t^{3}+t^{4}$.

\bigskip

\bigskip $%
\begin{array}{llll}
\mathcal{K} & \sigma  & \gamma  & KF(\mathcal{K};q,t) \\
8_{19} & 6 & -2 & q^{\underline{4}}-q^{\underline{2}%
}t^{1}+t^{2}-q^{2}t^{3}+q^{4}t^{4} \\
&  &  &  \\
9_{42} & 2 & 0 & q^{\underline{2}}t^{1}+q^{2}t^{3} \\
&  &  &  \\
10_{124} & 8 & -3 & -q^{\underline{6}}t^{\underline{1}}+q^{\underline{4}}-q^{%
\underline{2}}t^{1}-t^{1}+t^{2}-q^{2}t^{3}+q^{4}t^{4}-q^{6}t^{5} \\
&  &  &  \\
10_{128} & 6 & -2 & 2q^{\underline{4}}-q^{\underline{2}%
}t^{1}+t^{2}-q^{2}t^{3}+2q^{4}t^{4} \\
&  &  &  \\
10_{132} & 0 & 1 & -q^{\underline{2}}t^{1}-t^{1}-q^{2}t^{3} \\
&  &  &  \\
10_{136} & 2 & 0 & q^{\underline{2}}t^{1}+2t^{2}+q^{2}t^{3} \\
&  &  &  \\
10_{139} & 6 & -3 & -q^{\underline{6}}t^{\underline{1}}+q^{\underline{4}}-q^{%
\underline{2}}t^{1}-2t^{1}+t^{2}-q^{2}t^{3}+q^{4}t^{4}-q^{6}t^{5} \\
&  &  &  \\
10_{145} & -2 & 2 & q^{\underline{2}}-t^{1}+2t^{2}+q^{2}t^{2} \\
&  &  &  \\
10_{152} & -6 & 3 & -q^{\underline{6}}t^{\underline{3}}+q^{\underline{4}}t^{%
\underline{2}}-q^{\underline{2}}t^{\underline{1}}-q^{\underline{2}%
}+1-2t^{1}-q^{2}t-q^{2}t^{2}+q^{4}t^{2}-q^{6}t^{3} \\
&  &  &  \\
10_{153} & 0 & 0 & q^{\underline{4}}t^{\underline{2}}+q^{\underline{2}%
}+q^{2}t^{2}+q^{4}t^{2} \\
&  &  &  \\
10_{154} & 4 & -2 & q^{\underline{4}}+q^{\underline{2}}-q^{\underline{2}%
}t^{1}+2t^{1}+t^{2}+q^{2}t^{2}-q^{2}t^{3}+q^{4}t^{4} \\
&  &  &  \\
10_{161} & -4 & 2 & q^{\underline{4}}t^{\underline{2}}-q^{\underline{2}}t^{%
\underline{1}}+q^{\underline{2}}+1-q^{2}t^{1}+q^{2}t^{2}+q^{4}t^{2}%
\end{array}%
$

\bigskip

We also test the expression of $41$ homologically thick knots with $11$
crossings obtained in \cite{BG}.

\bigskip

$%
\begin{array}{llll}
\mathcal{K} & \sigma  & \gamma  & KF(\mathcal{K};q,t) \\
11n_{6} & 0 & 0 & q^{\underline{4}}t^{\underline{1}}+q^{\underline{2}}t^{%
\underline{1}}+2q^{\underline{2}}+q^{2}t^{1}+2q^{2}t^{2}+q^{4}t^{3} \\
&  &  &  \\
11n_{9} & 4 & -2 & q^{\underline{6}}t^{\underline{1}}+q^{\underline{4}}+q^{%
\underline{2}}-q^{\underline{2}%
}t^{1}+2t^{1}+t^{2}+q^{2}t^{2}-q^{2}t^{3}+q^{4}t^{4}+q^{6}t^{5} \\
&  &  &  \\
11n_{12} & 0 & -1 & -q^{\underline{2}}t^{\underline{1}}-2-t^{1}-q^{2}t^{1}
\\
&  &  &  \\
11n_{19} & -4 & 1 & -q^{\underline{4}}t^{\underline{2}}-t^{1}-q^{4}t^{2} \\
&  &  &  \\
11n_{20} & -2 & 0 & 2q^{\underline{2}}t^{\underline{1}}+2+2q^{2}t^{1} \\
&  &  &  \\
11n_{24} & 2 & 0 & q^{\underline{4}}+q^{\underline{2}%
}t^{1}+2t^{2}+q^{2}t^{3}+q^{4}t^{4} \\
&  &  &  \\
11n_{27} & 6 & -2 & q^{\underline{4}}+q^{\underline{6}}t^{\underline{1}%
}+t^{2}+q^{4}t^{4}+q^{6}t^{5} \\
&  &  &  \\
11n_{31} & 2 & -2 & q^{\underline{4}}t^{\underline{1}}+q^{\underline{2}}t^{%
\underline{1}}+q^{\underline{2}}+2-t^{1}+q^{2}t^{1}+q^{2}t^{2}+q^{4}t^{3} \\
&  &  &  \\
11n_{34} & 0 & 0 & q^{\underline{4}}t^{\underline{2}}+q^{\underline{4}}t^{%
\underline{1}}+q^{\underline{2}}t^{\underline{1}}+q^{\underline{2}%
}+q^{2}t^{1}+q^{2}t^{2}+q^{4}t^{2}+q^{4}t^{3} \\
&  &  &  \\
11n_{38} & 2 & 0 & q^{\underline{2}}t^{1}+t^{1}+q^{2}t^{3} \\
&  &  &  \\
11n_{39} & 0 & 0 & 2q^{\underline{2}}+4t^{1}+2t^{2}+2q^{2}t^{2} \\
&  &  &  \\
11n_{42} & 0 & 0 & q^{\underline{2}}t^{\underline{1}}+q^{\underline{2}%
}+2+2t^{1}+q^{2}t^{1}+q^{2}t^{2} \\
&  &  &  \\
11n_{45} & 0 & 0 & q^{\underline{4}}t^{\underline{1}}+q^{\underline{4}}+2q^{%
\underline{2}}+2t^{1}+2q^{2}t^{2}+q^{4}t^{3}+q^{4}t^{4} \\
&  &  &  \\
11n_{49} & 0 & 0 & q^{\underline{2}}t^{1}+2t^{1}+q^{2}t^{3}%
\end{array}%
$

\bigskip

$%
\begin{array}{llll}
\mathcal{K} & \sigma  & \gamma  & KF(\mathcal{K};q,t) \\
11n_{57} & 6 & -2 & q^{\underline{6}}t^{\underline{1}}+q^{\underline{4}}-q^{%
\underline{2}}t^{1}+t^{1}+t^{2}-q^{2}t^{3}+q^{4}t^{4}+q^{6}t^{5} \\
&  &  &  \\
11n_{61} & 4 & -1 & -q^{\underline{6}}t^{\underline{1}}-q^{\underline{4}}-q^{%
\underline{2}}t^{1}-t^{1}+t^{2}-q^{2}t^{3}-q^{4}t^{4}-q^{6}t^{5} \\
&  &  &  \\
11n_{67} & 0 & 0 & q^{\underline{2}}+q^{\underline{2}%
}t^{1}+2t^{1}+q^{2}t^{2}+q^{2}t^{3} \\
&  &  &  \\
11n_{70} & 4 & -1 & -q^{\underline{4}}-t^{1}-2t^{2}-q^{4}t^{4} \\
&  &  &  \\
11n_{73} & 0 & 0 & q^{\underline{4}}t^{\underline{1}}+q^{\underline{4}}+q^{%
\underline{2}}+q^{2}t^{2}+q^{4}t^{3}+q^{4}t^{4} \\
&  &  &  \\
11n_{74} & 0 & 0 & q^{\underline{2}}+2t^{1}+2t^{2}+q^{2}t^{2} \\
&  &  &  \\
11n_{77} & 6 & -3 & -q^{\underline{6}}t^{\underline{1}}+q^{\underline{4}%
}-2q^{\underline{2}}-q^{\underline{2}%
}t^{1}-4t^{1}+t^{2}-2q^{2}t^{2}-q^{2}t^{3}+q^{4}t^{4}-q^{6}t^{5} \\
&  &  &  \\
11n_{79} & 2 & 0 & 2q^{\underline{2}}t^{1}+2q^{2}t^{3} \\
&  &  &  \\
11n_{80} & -2 & 1 & -q^{\underline{4}}t^{\underline{2}}-q^{\underline{2}%
}-4t^{1}-q^{2}t^{2}-q^{4}t^{2} \\
&  &  &  \\
11n_{81} & 6 & -2 & q^{\underline{6}}t^{\underline{1}}+q^{\underline{4}}+q^{%
\underline{2}}t^{1}+t^{2}+q^{2}t^{3}+q^{4}t^{4}+q^{6}t^{5} \\
&  &  &  \\
11n_{88} & 6 & -2 & q^{\underline{6}}t^{\underline{1}}+q^{\underline{4}}-q^{%
\underline{2}}t^{1}+t^{2}-q^{2}t^{3}+q^{4}t^{4}+q^{6}t^{5} \\
&  &  &  \\
11n_{92} & -2 & 0 & q^{\underline{4}}t^{\underline{2}}+q^{\underline{2}}t^{%
\underline{1}}+q^{2}t^{1}+q^{4}t^{2} \\
&  &  &  \\
11n_{96} & 2 & 0 & q^{\underline{4}}+q^{\underline{2}}+q^{\underline{2}%
}t^{1}+q^{2}t^{2}+q^{2}t^{3}+q^{4}t^{4} \\
&  &  &  \\
11n_{97} & 0 & 0 & q^{\underline{2}}t^{\underline{1}}+q^{\underline{2}%
}+2t^{1}+q^{2}t^{1}+q^{2}t^{2}%
\end{array}%
$

\bigskip

$%
\begin{array}{llll}
\mathcal{K} & \sigma  & \gamma  & KF(\mathcal{K};q,t) \\
11n_{102} & -2 & 1 & -q^{\underline{2}}-t^{1}-2t^{2}-q^{2}t^{2} \\
&  &  &  \\
11n_{104} & 6 & -2 & q^{\underline{6}}t^{\underline{1}}+q^{\underline{4}}-q^{%
\underline{2}}t^{1}+2t^{1}+t^{2}-q^{2}t^{3}+q^{4}t^{4}+q^{6}t^{5} \\
&  &  &  \\
11n_{111} & 2 & -1 & -q^{\underline{4}}-q^{\underline{2}%
}-2t^{1}-q^{2}t^{2}-q^{4}t^{4} \\
&  &  &  \\
11n_{116} & 0 & 0 & q^{\underline{2}}t^{\underline{1}}+2t^{1}+q^{2}t^{1} \\
&  &  &  \\
11n_{126} & 6 & -2 & 3q^{\underline{4}}+t^{2}+3q^{4}t^{4} \\
&  &  &  \\
11n_{133} & 4 & -1 & -q^{\underline{6}}t^{\underline{1}}-2q^{\underline{4}%
}-q^{\underline{2}}t^{1}-t^{1}+t^{2}-q^{2}t^{3}-2q^{4}t^{4}-q^{6}t^{5} \\
&  &  &  \\
11n_{135} & 4 & -2 & q^{\underline{4}}t^{\underline{1}}+q^{\underline{2}}t^{%
\underline{1}}+q^{\underline{2}}-t^{1}+q^{2}t^{1}+q^{2}t^{2}+q^{4}t^{3} \\
&  &  &  \\
11n_{138} & 2 & 0 & 2q^{\underline{2}}t^{1}+2q^{2}t^{3} \\
&  &  &  \\
11n_{143} & 0 & 0 & q^{\underline{4}}t^{\underline{1}}+q^{\underline{2}}+q^{%
\underline{2}}t^{1}+q^{2}t^{2}+q^{2}t^{3}+q^{4}t^{3} \\
&  &  &  \\
11n_{145} & 0 & 0 & q^{\underline{4}}+q^{\underline{2}%
}+2t^{1}+q^{2}t^{2}+q^{4}t^{4} \\
&  &  &  \\
11n_{151} & 2 & -1 & -2q^{\underline{2}}-4t^{1}-2t^{2}-2q^{2}t^{2} \\
&  &  &  \\
11n_{152} & 2 & -1 & -q^{\underline{4}}t^{\underline{1}}-q^{\underline{4}%
}-2q^{\underline{2}}-2t^{1}-2q^{2}t^{2}-q^{4}t^{3}-q^{4}t^{4} \\
&  &  &  \\
11n_{183} & 4 & -2 & q^{\underline{4}}+2q^{\underline{2}}-q^{\underline{2}%
}t^{1}+2t^{1}+t^{2}+2q^{2}t^{2}-q^{2}t^{3}+q^{4}t^{4}%
\end{array}%
$

\end{document}